\documentclass[12pt,reqno]{amsart}

\usepackage[usenames, dvipsnames]{color}
\usepackage{amsmath, amsthm, amssymb}
\usepackage{amsfonts}
\usepackage[ansinew]{inputenc}
\usepackage{graphicx}
\usepackage[english]{babel}
\usepackage{enumerate}
\usepackage[bookmarksnumbered,colorlinks, linkcolor=blue, citecolor=red, pagebackref, bookmarks, breaklinks]{hyperref}

\definecolor{col}{rgb}{0.7,0.3,0.1}

\newtheorem{thm}{Theorem}[section]
\newtheorem{cor}[thm]{Corollary}
\newtheorem{lem}[thm]{Lemma}
\newtheorem{prop}[thm]{Proposition}
\newtheorem{defi}[thm]{Definition}

\def\pint{\operatorname {--\!\!\!\!\!\int\!\!\!\!\!--}}

\theoremstyle{remark}
\newtheorem{rem}[thm]{Remark}

\numberwithin{equation}{section}

\def\supp{\mathop{\text{\normalfont supp}}}
\def\tailg{\mathop{\text{\normalfont Tail}_g}}
\def\tailp{\mathop{\text{\normalfont Tail}_p}}
\def\tailpp{\mathop{\text{\normalfont Tail}_{p^+}}}
\def\tailpm{\mathop{\text{\normalfont Tail}_{p^-}}}

\def\osc{\mathop{\text{\normalfont osc}}}
\def\argmin{\mathop{\text{\normalfont Argmin}}}

\newcommand{\W}{\widetilde{W}}

\newcommand{\gfls}{(-\Delta_g)^s}

\newcommand{\R}{\mathbb{R}}
\newcommand{\N}{\mathbb{N}}

\newcommand{\ve}{\varepsilon}
\newcommand{\lam}{\lambda}
\newcommand{\intr}{\iint_{\R^n\times\R^n}}

\newcommand{\average}{{\mathchoice {\kern1ex\vcenter{\hrule height.4pt
width 6pt depth0pt} \kern-9.7pt} {\kern1ex\vcenter{\hrule
height.4pt width 4.3pt depth0pt} \kern-7pt} {} {} }}

\def\R{\mathbb{R}}

\begin{document}

\title[Interior and boundary regularity for the fractional $g$-Laplacian]{Interior and up to the boundary regularity for the fractional $g$-Laplacian: the convex case}

\author{Juli\'an Fern\'andez Bonder, Ariel Salort and Hern\'an Vivas}

\address[JFB and AS]{Instituto de C\'alculo (IC), CONICET\\
Departamento de Matem\'atica, FCEN - Universidad de Buenos Aires\\
Ciudad Universitaria, Pabell\'on I, C1428EGA, Av. Cantilo s/n\\
Buenos Aires, Argentina}

\email[JFB]{jfbonder@dm.uba.ar}

\email[AS]{asalort@dm.uba.ar}

\address[HV]{Centro Marplatense de Investigaciones matem\'aticas/Conicet, Dean Funes 3350, 7600 Mar del Plata, Argentina}

\email{havivas@mdp.edu.ar}



\begin{abstract}
We establish interior and up to the boundary H\"older regularity estimates for weak solutions of the Dirichlet problem for the fractional $g-$Laplacian with bounded right hand side and $g$ convex. These are the first regularity results available in the literature for integro-differential equations in the context of fractional Orlicz-Sobolev spaces.  
\end{abstract}

\maketitle

\section{Introduction and main results}

The aim of this work is to prove interior and up to the boundary H\"older regularity for solutions of the Dirichlet problem for the fractional $g-$Laplacian:
\begin{equation}\label{eq.dir}
\left\{ \begin{array}{cccl}
\gfls u  & = & f &  \textrm{ in }\Omega \\
u & = & 0  &\textrm{ in }\R^n\setminus\Omega,
\end{array} \right.
\end{equation}
where $\Omega$ a bounded open subset of $\R^n$ with $C^{1,1}$ boundary, $f\in L^\infty(\Omega)$ and $\gfls$ is the fractional $g-$Laplacian defined in \cite{FBS}: 
\begin{equation}\label{eq.gfls}
\gfls u(x):=2\, \textrm{p.v.}\int_{\R^n}g\left(\frac{u(x)-u(y)}{|x-y|^s}\right)\frac{dy}{|x-y|^{n+s}}
\end{equation}
with $g=G'$ the derivative of a Young function  $G$ (see Section \ref{sec.prel} for definitions) and p.v. stands for the principal value.

Fractional  operators arise naturally in the study of L\'evy processes with jumps, where the infinitesimal generator of a stable pure jump process is given  by a nonlocal operator. These processes have attracted much interest both in the PDE and Probability communities, since model  a wide range of phenomena in Physics, Finance, Image processing, or Ecology; see \cite{Ap04,CT16} and references therein.

Although the development of some aspects of the theory of nonlocal operators from the point of view of Analysis can be traced back several decades, either from a Harmonic Analysis or a potential theoretic approach, it was the seminal program developed by Caffarelli and Silvestre in \cite{CS1, CS2, CS3} that opened up the way for a (so to speak) modern regularity theory for elliptic nonlocal equations. Indeed, they developed a regularity theory for fully nonlinear nonlocal uniformly elliptic equations in a parallel fashion to that of the nowadays classic second order theory. 
The interested reader is referred to \cite{BV16, Ro1} further discussions and references.

The interior regularity results in Caffarelli and Silvestre's works relied heavily on a notion of ellipticity that constrains the behavior of the kernels of the operators to have certain smoothness and/or decay. These works were extended 
in \cite{K, Se1} 
to the class of \emph{rough} kernels, which can be highly oscillatory and irregular and further developments where made in 
\cite{CD}, where the symmetry assumption on the kernels is dropped. Also, non-translation invariant operators were considered in 
\cite{JX}. A parallel line of work was developed 
in \cite{RS1}, where it is  established   interior regularity estimates for general stable operators.
It is also worth highlighting the work 
\cite{Y} related with the Calder\'on-Zygmund theory for nonlinear  nonlocal operators
 
Once the interior regularity is understood, the next step is to study the the boundary regularity of solutions, which presents many differences  with respect to the case of local second order equations. Some of the most important advances in that direction are Ros-Oton and Serra's cited paper \cite{RS1} and its fully nonlinear counterpart \cite{RS2}. In these works they study Dirichlet problems of the form
\begin{equation}\label{pb}
\left\{ \begin{array}{rcll}
L u &=&f&\textrm{in }\Omega \\
u&=&0&\textrm{in }\R^n\backslash\Omega,
\end{array}\right.
\end{equation}
with $L$ a stable operator. Their main results are that if $\Omega$ is $C^{1,1}$ and $f\in L^\infty(\Omega)$, then $u/d^s\in C^{s-\varepsilon}(\overline\Omega)$ for all $\varepsilon>0$ for the linear case and if $\Omega$ is $C^{2,\gamma},\:a\in C^{1,\gamma}(S^{n-1})$, then if $f\in C^\gamma(\overline\Omega)$ then $u/d^s\in C^{\gamma+s}(\overline\Omega)$ for $\gamma\in(0,s)$ whenever $\gamma+s \notin \N$ for the nonlinear case. Here $d(x)$ is the distance to $\partial\Omega$ and $a$ is the spectral measure of $L$. 

On the other hand, using a rather different approach Grubb proved in \cite{GG15,GG14}, for a problem like \eqref{pb}, that if $\Omega$ is $C^\infty$ and $a\in C^\infty(S^{n-1})$, then
\[\quad f\in C^\gamma(\overline\Omega)\quad \Longrightarrow\quad u/d^s\in C^{\gamma+s}(\overline\Omega)\qquad \textrm{for all}\ \gamma>0,\]
whenever $\gamma+s\notin\mathbb Z$ for elliptic pseudodifferential operators satisfying the $s$-transmission property. In particular, $u/d^s\in C^\infty(\overline\Omega)$ whenever $f\in C^\infty(\overline\Omega)$. Moreover, when $s+\gamma \in\mathbb{N}$, more information is given in \cite{GG14} in terms of H\"older-Zygmund spaces $C^k_*$.

A common feature of all of the above cited papers is that they deal with the \emph{uniformly elliptic} case, in which, roughly speaking, the operator under study lies between two multiples of the \emph{fractional Laplacian}:
\[
(-\Delta)^su(x) \cong 2\,\text{p.v.}\int_{\R^n} \frac{u(x)-u(y)}{|x-y|^{n+2s}}dy
\]
arguably the most canonical example of nonlocal operators. The archetypal example of a \emph{degenerate elliptic} operator is given by the \emph{fractional $p-$Laplacian}: 
\[
(-\Delta)_p^su(x) = 2 \,\text{p.v.}\int_{\R^n} \frac{|u(x)-u(y)|^{p-2}(u(x)-u(y))}{|x-y|^{n+ps}}dy,\quad p>2.
\]
This operator arises, for instance, when studying minimizers of the  seminorm
\begin{equation}\label{eq.gag}
[u]_{s,p}:=\iint \frac{|u(x)-u(y)|^p}{|x-y|^{n+ps}}\:dxdy
\end{equation}
which are distinctive of the well-known fractional Sobolev spaces, see  \cite{DNPV}. 

Such scenario has been investigated thoroughly in the last years. Just to cite some relevant examples,  in \cite{dicastro} it was 
proved a   H\"older regularity result 
and a Harnack inequality in  \cite{dicastro2}. Higher H\"older estimates were   obtained in \cite{BrLiSch} and
higher integrability type results in \cite{BrLi}.  The eigenvalue problem was addressed in \cite{LiLi}. These papers concern the interior regularity of solutions; regarding global regularity, the most important work available is the paper \cite{IMS} 
on global H\"older regularity for the fractional $p-$Laplacian, which has been a mayor source of inspiration for our manuscript.

A rather natural question of interest is to replace the power growth in \eqref{eq.gag} by another type of behavior; this immediately gives way to consider  the fractional order Orlicz-Sobolev spaces, introduced by the first and the second author in \cite{FBS}. These spaces have received a great amount of attention in the last years, either in the study of their structural properties \cite{BOT, Cianchi, Cianchi2} or the PDE questions that arise such as eigenvalue problems \cite{S, SV} and P\'olya-Szeg\"o type results \cite{DNFBS}.

However, to the best of the authors' knowledge \emph{there are no results at all available in the literature dealing with regularity issues in this context}. The main goal of this paper is, therefore, to fill that gap. The main result we prove is global 
H\"older regularity for solutions of \eqref{eq.dir}. This is achieved under a 
suitable assumption on the smoothness of $\partial \Omega$ and a structural ``ellipticity'' assumption on $g$ 
(see equation \eqref{L} below). The result is stated as follows: 
\begin{thm}\label{thm.main}
There exist $\alpha\in(0,s]$ depending only on $n$, $s$, $\Omega$, $\Lambda$ and $\lambda$ such that for any weak solution $u\in W^{s,G}_0(\Omega)$ of \eqref{eq.dir}, $u\in C^\alpha(\overline\Omega)$ and
\begin{equation} \label{global.bound}
\| u \|_{C^\alpha (\overline\Omega)}\leq  C
\end{equation}
where $C$ is a constant depending on $s,n,\lambda,\Lambda,\|f\|_{L^\infty(\Omega)}$ and $\Omega$. 
\end{thm}

\medskip

\begin{rem}
In order to have the interior regularity, there is no need to require the so-called {\em complementary condition} $u=0$ in $\R^n\setminus\Omega$. As is customary in the nonlocal literature, it suffices for $u$ to verify $(-\Delta_g)^s u = f$ in $\Omega$ and some integrability conditions at infinity. See Section  3 (and Theorem \ref{thm.osc}) for a precise statement.
\end{rem}

\medskip

The strategy of the proof of Theorem \ref{thm.main} is divided in two steps:
\begin{enumerate}

\item first tackle the interior regularity of solutions. This is done by first proving an appropriate weak Harnack inequality and then an oscillation decay argument; 

\item the boundary regularity,  is achieved by constructing barrier functions that behaves like the distance function close to $\partial\Omega$; it is possible since the one dimensional profile $(x_n)_+^s$ satisfies $\gfls(x_n)_+^s=0$ so  straightening the boundary it gives that $\gfls d^s$ is bounded.

\end{enumerate}

We would like to point out that, in contrast with \cite{IMS}, the context of Orlicz spaces deals with nonhomogeneous operators; therefore, even if the overall strategy for proving regularity is the same as for the $p-$Laplacian (and for the fractional Laplacian for that matter), there are several nontrivial difficulties to overcome, see for instance the proof of Thereom \ref{thm.osc}.

The   paper is organized as follows: in Section \ref{sec.prel} we give   preliminary notions and definitions  used in the paper. In Section \ref{sec.wharn} we prove a weak Harnack inequality, from which an oscillation decay result  and the corresponding interior regularity estimates follow. In Section \ref{sec.bdry} we investigate the behavior near the boundary of solutions to \eqref{eq.dir} and   we prove that they are controlled by the distance function in  $\Omega$; putting together the results from Sections \ref{sec.wharn} and \ref{sec.bdry}, the global H\"older regularity follows and Theorem \ref{thm.main} is proven in Section \ref{sec.main}. Finally, in the Appendix we gather  the proof of some technical results  used in  the paper.   

\section{Preliminaries}\label{sec.prel}

\subsection{Young functions}\label{ssec.yfun}
An application $G\colon[0,\infty)\longrightarrow [0,\infty)$ is said to be a  \emph{Young function} if it admits the integral representation $G(t)=\int_0^t g(\tau)\,d\tau$, where the right continuous function $g$ defined on $[0,\infty)$ has the following properties:
\begin{align}
  &g(0)=0, \quad g(t)>0 \text{ for } t>0,   \label{g1} \tag{$g_1$} \\
&g \text{ is nondecreasing on } (0,\infty), \label{g2} \tag{$g_2$}\\
&\lim_{t\to\infty}g(t)=\infty. \label{g3} \tag{$g_3$}
\end{align}
From these properties it is easy to see that a Young function $G$ is continuous, nonnegative, strictly increasing and convex on $[0,\infty)$. Further, we recall that we may extend $g$ to the whole $\R$ in an odd fashion: for $t<0$ $g(t)=-g(-t)$.
 
We will consider the class of Young functions such that $g=G'$ is an absolutely continuous function that satisfies the condition 
\begin{equation} \label{L} 
1<\lambda \leq \frac{tg'(t)}{g(t)}\leq \Lambda<\infty.
\end{equation}

By a simple integration,   condition \eqref{L} implies that $G$ verifies
\begin{equation}\label{eq.p}
2<p^-\leq  \frac{tg(t)}{G(t)} \leq p^+<\infty,
\end{equation}
where $p^-=\lambda+1$ and $p^+=\Lambda+1$. In \cite[Theorem 4.1]{KR} it is shown that the upper bound in \eqref{L} (or in \eqref{eq.p}) is equivalent to the so-called \emph{$\Delta_2$ condition (or doubling condition)}, namely
\begin{equation} \label{delta2} \tag{$\Delta_2$}
g(2t) \leq 2^{\Lambda} g(t), \qquad G(2t)\leq 2^{p^+} G(t) \qquad t\geq 0.
\end{equation}
Further, since \eqref{L} implies $G''(t)>0$ we have that $G$ is a convex function. For for our purposes we will make the further assumption on $g$:
\begin{equation} \label{g4} \tag{$g_4$}
g \text{ is a convex function on } (0,\infty),
\end{equation}
which is analogous of  the degenerate case $p\ge 2$ for the fractional $p-$Laplacian. Throughout the paper, a constant will be called \emph{universal} if it depends only on $n,s,\lambda$ and $\Lambda$. Also, from now on,  assumption \eqref{g4} will be enforced. The convexity of $G$ and $g$ gives the following (see \cite[Lemma 2.1]{FBPS}):
\begin{lem}
For $\alpha\in[0,1]$ and  $t\geq 0$ it holds  that $G(\alpha t) \leq \alpha G(t)$, $g(\alpha t) \leq \alpha g(t)$; for $\alpha\geq 1$ and  $t\geq 0$ it holds that $G(\alpha t)\geq \alpha G(t)$, $g(\alpha t)\geq \alpha g(t)$. More generally,  for any, $\alpha,t\geq 0$
\begin{equation}\label{minmax1} 
\min\{\alpha^\lambda, \alpha^\Lambda\} g(t) \leq g(\alpha t) \leq \max\{\alpha^\lambda, \alpha^\Lambda\}g(t),
\end{equation}
\begin{equation}\label{minmax2} 
G(t)\min\{\alpha^{p^-}, \alpha^{p^+}\} \leq G(\alpha t) \leq G(t)\max\{\alpha^{p^-}, \alpha^{p^+}\}.
\end{equation}
\end{lem}

\subsection{Fractional Orlicz-Sobolev spaces}

Recall that for a Young function $G$ the Orlicz spaces are given by
\[
L^G(\Omega):=\left\{u:\Omega\longrightarrow\R,\text{ measurable}\colon \int_\Omega G(u(x))\:dx<\infty\right\}.
\] 
These are Banach spaces under the assumption that condition \eqref{delta2} holds.

Given $s\in (0,1)$ we will work with the fractional Orlicz-Sobolev spaces
$$
W^{s,G}(\Omega):=\left\{ u\in L^G(\Omega)\colon D_s u \in L^G (\Omega\times\Omega, d\mu)\right\}
$$
where $d\mu=|x-y|^{-n}\,dxdy$ and the  $s-$H\"older quotient is denoted as
\begin{equation}\label{eq.shold}
D_s u(x,y):=\frac{u(x)-u(y)}{|x-y|^s}.
\end{equation}

Over the  space $W^{s,G}(\Omega)$ we define the Luxemburg type norm 
\begin{equation}\label{eq.norm}
\|u\|_{s,G,\Omega} := \|u\|_{G,\Omega} + [u]_{s,G,\Omega},
\end{equation}
where
\begin{align*}
\|u\|_{G,\Omega} & :=\inf\left\{ \lam>0\colon \int_\Omega G\left(\frac{u(x)}{\lam}\right)\,dx \leq 1\right\} \\
[u]_{s,G,\Omega} & :=\inf\left\{\lambda>0\colon \iint_{\Omega\times \Omega} G\left(\frac{D_s u}{\lambda}\right)\,d\mu \le 1\right\}.
\end{align*}

Given $\Omega\subset \R^n$ bounded we define
$$
\widetilde W^{s,G}(\Omega):=\left\{u\in L^G_{loc}(\R^n) \colon \exists U\supset\supset \Omega \text{ s.t. } \|u\|_{s,G,U}+ \int_{\R^n} g\left(\frac{|u(x)|}{(1+|x|)^s}\right) \frac{dx}{(1+|x|)^{n+s}} <\infty \right\}.
$$

If $\Omega$ is unbounded
$$
\widetilde W_{\text{loc}}^{s,G}(\Omega):=\{u\in L^G_{loc}(\R^n) \colon u\in \widetilde W^{s,G}(\Omega') \text{ for any bounded }\Omega'\subset \Omega\}.
$$

Observe that, when $u\in L^\infty(\R^n)$
\begin{align*}
\int_{\R^n} g\left(\frac{|u(x)|}{(1+|x|)^s}\right) \frac{dx}{(1+|x|)^{n+s}} &
\leq C(\|u\|_\infty)
\int_0^\infty g\left(\frac{1}{(1+r)^s}\right) \frac{dr}{(1+r)^{1+s}}\\
&\leq C(\|u\|_\infty) G(1)
\int_0^\infty  (1+r)^{-s-1}\,dr <\infty,
\end{align*}
where we have used \eqref{delta2}, \eqref{L} and \eqref{g2}. Observe that a similar estimate can be obtained for functions that are not bounded but grow in a controlled way. For instance, a similar bound is obtained if $u$ verifies
$$
|u(x)|\le C (1+|x|)^{s+a}, \quad \text{for } a<\frac{s}{\Lambda}.
$$


Further, we define the spaces 
$$
W^{s,G}_0(\Omega)  :=\left\{u\in W^{s,G}(\R^n)\colon u=0\text{ in }\Omega^c\right\},\qquad 
W^{-s,\tilde{G}}(\Omega)  :=\left(W^{s,G}_0(\Omega)\right)^\ast
$$
where $\tilde G$ denotes the complementary Young function  of $G$ is defined as
\begin{equation}\label{eq.comp}
\tilde G(t):=\sup\{tw -G(w)\colon w>0\}.
\end{equation}
	
For all measurable $u\colon \R^n \to \R$ the \emph{nonlocal tail} centered at $x\in \R^n$   is defined as 
$$
\tailg(u;x,1) = g^{-1}\left( \int_{B_1^c(x)} g\left(  \frac{u(y)}{|x-y|^s}\right) \frac{dy}{|x-y|^{n+s}}\right),
$$
$$
\tailg(u;x,R) =  g^{-1}\left( R^s \int_{B_R^c(x)}   g\left(R^s  \frac{u(y)}{|x-y|^s}\right) \frac{dy}{|x-y|^{n+s}}\right).
$$

When $x=0$ we just write $\tailg(u;x,R)=\tailg(u;R)$. Moreover, then $G(t)=t^p$ for some $p\geq 2$ we write 
$$
\tailp(u;x,R)=\left( R^{sp} \int_{B_R^c(x)} \frac{|u(y)|^{p-1}}{|x-y|^{n+sp}}\,dy \right)^\frac{1}{p-1}
$$
and $\tailp(u;x,R)=\tailp(u;R)$ when $x=0$.

\subsection{Notions of solutions}

In this section we give the appropriate notions of solutions that will be used. We start with the definition of weak solution.
\begin{defi}\label{def.wsol}
Let $\Omega$ be a domain in $\R^n$, $u\in \W^{s,G}(\Omega)$ and $f\in W^{-s,\tilde{G}}(\Omega)$. We say that $u$ is a weak subsolution of $\gfls u=f$ in $\Omega$ if for any test function $\varphi\in W_0^{s,G}(\Omega)$ satisfying $\varphi\geq0$ a.e. in $\Omega$ there holds
\begin{equation}\label{eq.wsubsol}
\intr  g\left(\frac{u(x)-u(y)}{|x-y|^s}\right)\frac{\varphi(x)-\varphi(y)}{|x-y|^s}\,d\mu  \leq \int_{\Omega}f\varphi\,dx.
\end{equation}
We say that $u$ is a supersolution if $-u$ is a subsolution and that $u$ is a solution if it is both a sub and a supersolution. In particular, if $u$ is a solution
\begin{equation}\label{eq.wsol}
\intr  g\left(\frac{u(x)-u(y)}{|x-y|^s}\right)\frac{\varphi(x)-\varphi(y)}{|x-y|^s}\,d\mu=\int_{\Omega}f\varphi\,dx
\end{equation}
any test function $\varphi\in W_0^{s,G}(\Omega)$.
 
If $\Omega$ is unbounded, $u\in \widetilde{W}_{\text{loc}}^{s,G}(\Omega)$ is a weak subsolution, supersolution or solution of $\gfls u=f$ if \eqref{eq.wsubsol}-\eqref{eq.wsol} hold for any bounded open set $\widetilde{\Omega}\subset\Omega$. 
\end{defi}

Lemma \ref{lem.buenadefi} implies that the integrals in \eqref{eq.wsubsol} and \eqref{eq.wsol} are well defined.

Next, we define pointwise and strong solutions:
\begin{defi}\label{def.stsol}
Let $\Omega$ be a domain in $\R^n$, $u\in \widetilde{W}_{\text{loc}}^{s,G}(\Omega)$ and $f$ a measurable function in $\Omega$. We say that $u$ is a pointwise subsolution of $\gfls u=f$ in $\Omega$ if the limit 
\begin{equation}\label{eq.psol}
\limsup_{\varepsilon\rightarrow 0^+}2\int_{B^c_\varepsilon(x)} g\left(\frac{u(x)-u(y)}{|x-y|^s}\right) \frac{dy}{|x-y|^{n+s}}\leq f(x)
\end{equation}
for almost every Lebesgue point of $u$ in $\Omega$ (in particular, a.e. in $\Omega$). We say that $u$ is a pointwise supersolution if $-u$ is a subsolution and that $u$ is a solution if it is both a sub and a supersolution. 

Moreover, if $f\in L_{\text{loc}}^1(\Omega)$ we say that $u$ is a strong subsolution of $\gfls u=f$ if the limit \eqref{eq.psol} holds in $L^1_{\text{loc}}(\Omega)$. Analogous definitions hold for strong supersolution and solutions. 
\end{defi}

\section{Weak Harnack and H\"older regularity}\label{sec.wharn}

In this section we prove a weak Harnack inequality and, as a consequence, get interior $C^\alpha$ estimates for solutions of \eqref{eq.dir} (see Corollary \ref{cor.hold}). 

\begin{lem} \label{lema.1}
Let $\varphi \in C^2(\R^n)$ such that $\|\varphi\|_{C^2(\R^n)}<\infty$. Then 
$$
|(-\Delta_g)^s \varphi(x)| \leq K,
$$
where $K$ is a (computable) positive constant depending on $\|\varphi\|_{C^2(\R^n)}$, $n$, $s$, $g$ and $g'$.
\end{lem}

\begin{proof}
Let $\varphi\in C^2(\R^n)\cap L^\infty(\R^n)$ be fixed.
Denote 
$$
D_z \varphi(x):=\varphi(x)-\varphi(x+z),\qquad D_z^2 \varphi(x):= \varphi(x+z) -2\varphi(x) +\varphi(x-z).
$$
Let us write
\begin{align*}
(-\Delta_g)^s \varphi(x)&= 2 \left( \lim_{\ve \to 0^+} \int_{\ve<|x-y|<1} + \int_{|x-y|\geq 1} \right) g\left( \frac{\varphi(x)-\varphi(y)}{|x-y|^s} \right)   \frac{dy}{|x-y|^{n+s}}\\
&:=2 \lim_{\ve \to 0^+} (a_\ve)+(b).
\end{align*}
(The smoothness and boundedness of $\varphi$ already make $(-\Delta_g)^s \varphi(x)$ well defined.)

Taking into account that $g$ is nondecreasing we readily get that
$$
(b)\leq g(2\|\varphi\|_\infty) \int_{|x-y|\geq 1}|x-y|^{-n-s}\,dy =\frac{n\omega_n}{s} g(2\|\varphi\|_\infty).
$$
Expression $(a_\ve)$ can be rewritten as
\begin{align} \label{eq.1}
\begin{split}
(a_\ve)&= \int_{\ve<|z|<1}  g\left( \frac{D_z \varphi(x)}{|z|^s} \right)    \frac{dz}{|z|^{n+s}}= \int_{\ve<|z|<1}  g\left( \frac{D_{-z}\varphi(x)}{|z|^s} \right)    \frac{dz}{|z|^{n+s}}\\
&= \frac{1}{2}   \int_{\ve<|z|<1}  \left( g\left( \frac{D_z \varphi(x)}{|z|^s} \right) +  g\left( \frac{D_{-z} \varphi(x)}{|z|^s} \right)     \right) \frac{dz}{|z|^{n+s}}\\
&=-\frac12 \int_{\ve<|z|<1} (\psi(1)-\psi(0))\frac{dz}{|z|^{n+s}}=-\frac12 \int_{\ve<|z|<1}  \int_0^1 \psi'(t)\,dt \frac{dz}{|z|^{n+s}}\\
\end{split}
\end{align}
where we have denoted
$$
\psi(t)=   g\left( \frac{ (1-t) D_{z}\varphi(x) - t D_{-z} \varphi(x) }{|z|^s} \right).
$$
Since the derivative of $\psi$ can be computed to be 
$$
\psi'(t)=g'\left( \frac{ (1-t)D_{z}\varphi(x) - t D_{-z} \varphi(x)) }{|z|^s}  \right)  \cdot \frac{D_z^2 \varphi(x)}{|z|^s},
$$
and since we have $|D_z \varphi|\leq \| \nabla \varphi\|_\infty |z|\:, |D_z^2 \varphi| \leq \|D^2 \varphi\|_\infty |z|^2$ we get
$$
|\psi'(t)| \leq g'(2\|\nabla \varphi\|_\infty |z|^{1-s}) \|D^2 \varphi\|_\infty |z|^{2-s},
$$
where we have used that $g'$  increasing due to \eqref{g4}.
Then, expression \eqref{eq.1} can be bounded as follows
\begin{align*}
(a_\ve) &\leq \int_{\ve<|z|<1} g'(2\|\nabla \varphi\|_\infty |z|^{1-s}) \|D^2 \varphi\|_\infty |z|^{2-n-2s}\,dz\\ 
&\leq  g'(2\|\nabla \varphi\|_\infty) \|D^2 \varphi\|_\infty  n \omega_n\int_\varepsilon^1 r^{1-2s}\:ds = g'(2\|\nabla \varphi\|_\infty) \|D^2 \varphi\|_\infty  n \omega_n\frac{1-\ve^{2(1-s)}}{2(1-s)}.
\end{align*}

Finally, combining the bounds for $(a_\ve)$ and $(b)$ we get that
$$
|(-\Delta_g)^s \varphi(x)| \leq \frac{2n\omega_n}{s} g(2\|\varphi\|_\infty) +\frac{n \omega_n}{1-s} g'(2\|\nabla \varphi\|_\infty) \|D^2 \varphi\|_\infty  
$$
as desired.
\end{proof}

We are now in position to prove our weak Harnack inequality.

%
 
\begin{thm}[Weak Harnack inequality] \label{harnack}
If $u\in \W^{s,G}(B_{R/3})$ satisfies weakly 
\begin{align} \label{eq.w.harnack}
\begin{cases}
(-\Delta_g)^s u \geq -K &\quad \text{ in } B_{R/3}\\
u\geq 0 &\quad \text{ in } \R^n
\end{cases}
\end{align}
for some $K\geq 0$, then there exists universal $\sigma\in(0,1)$, and  $C_0>0$  such that
$$
\inf_{B_{R/4}} u \geq \sigma R^sg^{-1}\left( \pint_{B_R \setminus B_{R/2}} g(R^{-s}|u|)\,dx \right) - R^sg^{-1}\left(C_0 K\right).
$$
\end{thm}

\begin{proof}
We prove   the result for $R=1$ and then get \eqref{eq.w.harnack} by scaling from Lemma \ref{lem.scaling}.

Let $\varphi\in C^\infty(\R^n)$ be such that $0\leq \varphi\leq 1$ in $\R^n$, $\varphi=1$ in $B_{1/4}$ and $\varphi=0$ in $B_{1/3}^c$. Then $\|\varphi\|_\infty\leq 1$ and we can assume that $\|\nabla\varphi\|_\infty\leq 2$.  Then, in light of Lemma \ref{lema.1}
\begin{equation} \label{eqq.0}
|(-\Delta_g)^s \varphi| \leq C \quad \text{ in } B_{1/3}
\end{equation}
for some $C$ depending of $\|\varphi\|_{C^2}$, $g'$ and $s$.  
We define 
$$
L= g^{-1}\left( \pint_{B_1 \setminus B_{1/2}} g(|u|)\,dx \right)
$$
and for any $\sigma>0$  set $w=\sigma L \varphi +u\chi_{B_1\setminus B_{1/2}}$. Then,  by  Lemma \ref{nonlocal.behavior}, $(-\Delta_g)^s w(x) =(a)+(b)$ in $B_{1/3}$, where
\begin{align*}
(a)&=(-\Delta_g)^s (\sigma L\varphi)(x)\\
(b) &=   2 \int_{B_1\setminus B_{1/2}} \left( g\left(\frac{\sigma L \varphi(x)- u(y)}{|x-y|^s}\right) - g\left(\frac{\sigma L \varphi(x)}{|x-y|^s} \right) \right) \frac{dy}{|x-y|^{n+s}}.
\end{align*}

Expression $(a)$ can be bounded by using Lemma \ref{lema.1}, \eqref{delta2} and \eqref{L} as
\begin{align*}
(a)&\leq \frac{n\omega_n}{s} g(2\|\sigma L\varphi\|_\infty) + g'(2\|\nabla \sigma L\varphi \|_\infty) \|D^2 \sigma L\varphi\|_\infty  \frac{n \omega_n}{1-s}\\
&\leq C_1(n,s, \Lambda)(g( \sigma L ) + g'(  \sigma L ) \sigma L  ) \leq C_1 g(\sigma L).
\end{align*}

For expression $(b)$, by applying Lemma \ref{lema.0}, we get
\begin{align*}
(b) &\leq -2^{2-\Lambda} \int_{B_1\setminus B_{1/2}} g\left( \frac{|u(x)|}{|x-y|^s} \right) \frac{dy}{|x-y|^{n+s}}\leq -2^{1-\Lambda} \int_{B_1\setminus B_{1/2}} g\left(  |u(x)|  \right) \,dy\\
&\leq -2^{2-\Lambda} n\omega_n  (1-\tfrac{1}{2^n})g(L)\leq -C_2(n,s,\Lambda) g(L).
\end{align*}

Assume that $\sigma<1$, then from the bounds for $(a)$ and $(b)$ we get
$$
(-\Delta_g)^s w(x) \leq C_1\sigma g( L)    -C_2 g(L).
$$
Then, if we further assume $\sigma < \min\left\{1, \frac{C_2}{2C_1}   \right\}$, we get the upper estimate
\begin{equation} \label{cota}
(-\Delta_g)^s w(x) \leq -\frac{C_2}{2} g(L)   \quad \text{ in } B_{1/3}.
\end{equation}

We distinguish two cases:

$\bullet$ if $L \leq g^{-1}\left(\frac{2K}{C_2}\right)$, then 
$\inf_{B_{1/4}} u \geq 0 \geq \sigma L - g^{-1}\left(\frac{2K}{C_2} \right)$,

$\bullet$ if $L > g^{-1}\left(\frac{2K}{C_2}\right)$, then, from \eqref{cota}, 
$(-\Delta_g)^s w(x) \leq -K\leq (-\Delta_g)^s u$ in $B_{1/3}$,
moreover, by construction $w=\chi_{B_1\setminus B_{1/2}} u \leq u$ in $B^c_{1/3}$. In light of the comparison principle stated in Proposition \ref{compara}, the last two relations imply that
$$
\inf_{B_{1/4}} u \geq \sigma L \geq \sigma L -  g^{-1}\left(\frac{2K}{C_2} \right) 
$$
and the proof concludes.
\end{proof}

Next we extend Theorem \ref{harnack} to supersolutions  nonnegative in a ball.

\begin{lem}  \label{lemman}
There exists $\sigma\in(0,1)$, $C>0$, and for all $\ve>0$   a constant $C_\ve>0$ such that, if $u\in \widetilde W^{s,G}(B_{R/3})$ satisfies
\begin{align*}
\begin{cases}
(-\Delta_g)^s u \geq -K &\quad \text{ in } B_{R/3}\\
u\geq 0 &\quad \text{ in } B_R
\end{cases}
\end{align*}
for some $K\geq 0$, then
$$
\inf_{B_{R/4}} u \geq R^s \sigma g^{-1}\left( \pint_{B_R \setminus B_{R/2}} g(R^{-s}u)\,dx \right) -\tilde C R^s g^{-1}(K) - C_\ve \tailg(u_-;R) - \ve  \sup_{B_R}u  .
$$
\end{lem}
\begin{proof}
We prove first the result for $R=1$ and then  the result follows for any $R>0$ by scaling from Lemma \ref{lem.scaling}.

We apply Lemma \ref{nonlocal.behavior} to $v=u_-$, $u+v=u_+$ and $\Omega=B_{1/3}$. Then, in the weak sense in $B_{1/3}$ we have that
\begin{align*}
(-\Delta_g)^s u_+(x)&=(-\Delta_g)^s u(x)+ 2\int_{B_{1/3}^c} \left[g\left(\frac{ u(x) -u_+(y) }{|x-y|^s} \right) -  g\left(\frac{ u(x) -u(y) }{|x-y|^s} \right) \right] \frac{dy}{|x-y|^{n+s}} \\
&\geq -K + 
2\int_{u<0} \left[ g\left(\frac{ u(x) }{|x-y|^s} \right) -  g\left(\frac{ u(x) -u(y) }{|x-y|^s} \right) \right]\frac{dy}{|x-y|^{n+s}}\\
&\geq -K +C\int_{u<0} \left[ g\left(\frac{ u(x) }{|y|^s} \right) -  g\left(\frac{ u(x) -u(y) }{|y|^s} \right) \right]\frac{dy}{|y|^{n+s}}
\end{align*}
where in the last inequality we have used that $|x-y|>\tfrac23|y|$.

By using Lemma \ref{lema.separa}, for any $\theta>0$ there exists $C_\theta>0$ such that, since $g$ is odd
$$
g\left(\frac{ u(x) }{|y|^s} \right) -  g\left(\frac{ u(x) -u(y) }{|y|^s} \right) \geq c_\theta g\left(\frac{ u(x) }{|y|^s} \right) + C_\theta g\left(\frac{u(y) }{|y|^s} \right)
$$
with $c_\theta=1-(1-\theta)^\Lambda<1$. Then
\begin{align*}
(-\Delta_g)^s u_+(x)&\geq 
-K - \int_{u<0} g\left(\frac{ u(x) }{|y|^s} \right)\frac{dy}{|y|^{n+s}} -   C_\theta \int_{u<0}  g\left(\frac{u(y) }{|y|^s} \right)\frac{dy}{|y|^{n+s}}\\
&\geq 
-K - \int_{u<0} g\left(\frac{ \sup_{B_1} u }{|y|^s} \right)\frac{dy}{|y|^{n+s}} -     C_\theta g(\tailg(u_-;1))\\
&\geq 
-K -  g\Big( \sup_{B_1} u \Big)  \int_{u<0}  \frac{dy}{|y|^{n+s(\Lambda+1)}} -     C_\theta g(\tailg(u_-;1))\\
&\geq 
-K - C(\Lambda) g\Big( \sup_{B_1} u \Big)   -     C_\theta g(\tailg(u_-;1)):=-\tilde K.
\end{align*}
Now, we can apply Theorem \ref{harnack} to $u_+$, 
\begin{align*}
\inf_{B_{1/4}} u &\geq \sigma g^{-1}\left( \pint_{B_1 \setminus B_{1/2}} g(u)\,dx \right) - g^{-1}\left(C_0 \tilde K \right).
\end{align*}

Observe that from Lemma \ref{delta2.inversa}, for $\theta< \frac{\ve^\lam}{4CC_0}\leq 1$, we get
\begin{align*}
g^{-1}&\left(C_0 \tilde K \right) 
\leq 2^\frac{2}{\lambda}\left[
g^{-1}(C_0 K) + g^{-1}\Big(C_0 C \theta g\Big( \sup_{B_1} u\Big) \Big) + g^{-1}(C_0   C_\theta g(\tailg(u_-;1)) )  \right]\\
&\leq 2^\frac{2}{\lambda}  \Big[ \max\{1, C_0^\frac{1}{\lam} \}
g^{-1}(K)  , + ,  \max\{1, (C_0C \theta)^\frac{1}{\lam} \}  \sup_{B_1} u  \max\{1,(C_0 C_\theta)^\frac{1}{\lam}\} \tailg(u_-;1) \Big] \\
&\leq \tilde C g^{-1}(K) + C_\ve \tailg(u_-;1) +  \ve    \sup_{B_1}u.
\end{align*}
Therefore
$$
\inf_{B_{1/4}} u \geq \sigma g^{-1}\left( \pint_{B_1 \setminus B_{1/2}} g(u)\,dx \right) -\tilde C g^{-1}(K) - C_\ve \tailg(u_-;1) -  \ve\sup_{B_1}u
$$
as required.
\end{proof}

As a consequence of the weak Harnack inequality, we get the following oscillation decay result:

\begin{thm}\label{thm.osc}
There exist $\alpha\in (0,1)$ and a universal constant $C>0$  with the following property: if $R_0\in(0,1)$ and $u\in \W^{s,G}(B_{R_0})\cap L^\infty(B_{R_0})$ satisfies  
\[
|\gfls u|\leq K
\]
weakly in $B_{R_0}$ and that $\tailpp(u;R_0)<\infty$, then for all $r\in(0, R_0)$ it holds that
$$
\osc_{B_r} u \leq   C(R_0^{s-\alpha} g^{-1}(K)+ R_0^{-s}Q(u;R_0))    r^\alpha
$$
where, for some $\beta=\beta(\Lambda,\lambda)>0$
$$
Q(u;R_0)=  (\|u\|_{L^\infty(B_0)} + \tailpp(u;R_0)+ \tailpm(u;R_0))^\beta.
$$
\end{thm}

%

\begin{proof}
Given $R_0\in (0,1)$, for any $j\in\N_0$ set the quantities
$$
R_j=\frac{R_0}{4^j},\quad B_j = B_{R_j}, \quad \frac12 B_j=B_{R_j/2}.
$$
Let us prove that there exists a universal constant $\alpha\in (0,1)$,  $L>0$, a nondecreasing sequence $\{m_j\}$ and  a nonincreasing sequence $\{M_j\}$ such that, for all $j\geq 0$, 
$$
m_j \leq \inf_{B_j} u \leq \sup_{B_j} u \leq M_j, \qquad M_j-m_j=LR^\alpha_j.
$$

We proceed by induction on $j$. 

\medskip

{\bf Step 1.} When $j=0$, set $M_0=\sup_{B_0} u$, $m_0=M_0-LR_0^\alpha$, where $L>0$ satisfies
\begin{equation} \label{paso.0}
L\geq \frac{2\|u\|_{L^\infty(B_0)}}{R_0^\alpha},
\end{equation}
which implies that $m_0 \leq \inf_{B_0}u\leq M_0$.

\medskip

{\bf Step 2.} Inductive step: assume that sequences $\{m_j\}$ and $\{M_j\}$ are constructed up to the index $j$. Then
\begin{align*}
M_j-  m_j &= \pint_{B_j\setminus \frac12 B_j} (M_j-u)\,dx +   \pint_{B_j\setminus \frac12 B_j} (u-m_j)\,dx
\end{align*}
by using the convexity of $g$, 
\begin{align*}
g(R_j^{-s}(M_j-  m_j)) &\leq C\pint_{B_j\setminus \frac12 B_j} g(Rj^{-s}(M_j-u))\,dx +   C\pint_{B_j\setminus \frac12 B_j} g(R_j^{-s}(u-m_j))\,dx 
\end{align*}
from where
\begin{align*}
M_j - m_j \leq & C R_j^s g^{-1}  \left(  \pint_{B_j\setminus \frac12 B_j} g(R_j^{-s}(M_j-u))\,dx \right)+ C R_j^s g^{-1}\left(  \pint_{B_j\setminus \frac12 B_j} g(R_j^{-s}(u-m_j))\,dx \right).
\end{align*}

Let $\sigma\in(0,1)$, $\tilde C>0$ be as in Lemma \ref{lemman}. From the previous inequality and Lemma \ref{lemman} we get
\begin{align*}
C &\sigma(M_j - m_j) 
\leq 
\sigma   R_j^{s} g^{-1} \left(  \pint_{B_j\setminus \frac12 B_j} g(R_j^{-s}(M_j-u))\,dx \right)\\
&\quad + \sigma   R_j^s g^{-1}\left( \frac{1}{R_j^s} \pint_{B_j\setminus \frac12 B_j} g(R_j^{-s}(u-m_j))\,dx \right)\\
&\leq 
\inf_{B_{j+1}}(M_j-u) +  \inf_{B_{j+1}}(u-m_j) +  2\tilde C R_j^s g^{-1}(K)+ C_\ve [ \tailg((M_j-u)_-;R_j)\\
&\quad +\tailg((u-m_j)_-;R_j)]   +   \ve    \left[  \sup_{B_{R_j}}(M_j-u) +  \sup_{B_{R_j}}(u-m_j) \right].
\end{align*}

Setting $\ve=\sigma/4$, $C=\max\{2\tilde C,C_\ve\}$ and rearranging terms we get
\begin{align} \label{cota.osc}
\begin{split}
\osc_{B_{j+1}} u &\leq  \frac{2-\sigma}{2} (M_j-m_j)+ C[ (R_0^s g^{-1}(K) + \tailg((M_j-u)_-;R_j) + \tailg((u-m_j)_-;R_j)).
\end{split}
\end{align}

\medskip

{\bf Step 3.} Let us estimate the tails. Observe that we can write
\begin{align} \label{tail.cota}
\begin{split}
R_j^{-s} g(\tailg((u-m_j)_-;R_j)) &= \sum_{k=0}^{j-1}\int_{B_k\setminus B_{k+1}} g\left( R_j^{s} \frac{(u(y)-m_j)_-}{|y|^s}\right) \frac{dy}{|y|^{n+s}}\\ &\quad + \int_{B_0^c} g\left( R_j^{s} \frac{(u(y)-m_j)_-}{|y|^s}\right) \frac{dy}{|y|^{n+s}}.
\end{split}
\end{align}
Let us deal with the first term. By inductive hypothesis, for all $0\leq k \leq j-1$ we have in $B_k\setminus B_{k+1}$
$$
(u-m_j)_- \leq m_j-m_k\leq (m_j-M_j)+(M_k-m_k)=L (R_k^\alpha - R_j^\alpha),
$$
hence
\begin{align*}
(i):=\sum_{k=0}^{j-1}\int_{B_k\setminus B_{k+1}}& g\left( R_j^{s} \frac{(u(y)-m_j)_-}{|y|^s}\right) \frac{dy}{|y|^{n+s}} 
\leq 
\sum_{k=0}^{j-1}\int_{B_k\setminus B_{k+1}} g\left( L R_j^{\alpha+s}  \frac{(4^{\alpha(j-k)}-1)}{|y|^s}\right) \frac{dy}{|y|^{n+s}}\\
&\leq
\sum_{k=0}^{j-1} g\left( L R_j^{\alpha+s} R_{k}^{-s} (4^{\alpha(j-k)}-1)\right)  \int_{B_k\setminus B_{k+1}} \frac{dy}{|y|^{n+s}}\\
&\leq
C(s,n)  g\left( L R_j^{\alpha+s}  R_j^{-s} \right) R_0^{-s} \sum_{k=0}^{j-1} (4^{\alpha(j-k)}-1)^\Lambda 4^{sk}.
\end{align*}
Then, if we choose $\alpha<\frac{s}{\Lambda}$
\begin{align*}
(i)
&\leq C(s,n)  g\left( L R_j^{\alpha+s} R_0^{-s}  R_k^{-s} \right)  4^{sj} \sum_{k=0}^{j-1} (4^{\alpha(j-k)}-1)^\Lambda 4^{-s(j-k)}\\
&\leq C(s,n)  g\left( L R_j^{\alpha}  \right) R_0^{-s}  R_j^{-s} \sum_{h=1}^\infty (4^{h\alpha}-1)^\Lambda 4^{-sh}\leq C(s,n)  g\left( L R_j^{\alpha}  \right)R_0^{-s}  R_j^{-s} S(\alpha)
\end{align*}
where
$$
S(\alpha) := \sum_{h=1}^\infty (4^{ h\alpha} -1)^\Lambda 4^{-sh} \leq  \sum_{h=1}^\infty 4^{ h[\alpha\Lambda-s]} = \frac{1}{1-4^{\alpha\Lambda-s}}<\infty.
$$
Moreover, observe that $S(\alpha)\to 0$ as $\alpha\to 0^+$. Therefore
\begin{equation} \label{cota.1}
g^{-1}(R_j^{s} \cdot (i)) \leq C R_0^{-\frac{s}{\Lambda}} S(\alpha)^\frac{1}{\Lambda} L R_j^\alpha.
\end{equation}

Let us deal with the second integral. Since by inductive hypothesis
$$
m_j\leq \inf_{B_j}u \leq \sup_{B_j}u \leq \|u\|_{L^\infty(B_0)},
$$
we have 
\begin{align*}
(ii):=\int_{B_0^c} &g\left( R_j^{s} \frac{(u(y)-m_j)_-}{|y|^s}\right) \frac{dy}{|y|^{n+s}} 
\leq 
\int_{B_0^c} g\left( R_j^{s} \frac{|u(y)|+ \|u\|_{L^\infty(B_0)}}{|y|^s}\right) \frac{dy}{|y|^{n+s}} \\
&\leq 
C \int_{B_0^c} g\left( R_j^{ s} \frac{ |u(y)|}{|y|^s}\right) \frac{dy}{|y|^{n+s}} + C \int_{B_0^c} g\left( R_j^{ s} \frac{ \|u\|_{L^\infty(B_0)}}{R_0^s}\right) \frac{dy}{|y|^{n+s}}\leq (a) + (b).
\end{align*} 
Expression $(a)$ can be bounded by using \eqref{delta2} as

\begin{align*}
(a) &\leq 
g( R_j^{ s}R_0^s R_0^{-s} ) \left( \int_{B_0^c\cap\{y\colon u(y)|y|^{-s}\geq 1\}}\frac{  |u(y)|^\Lambda}{|y|^{n+sp^+}} \,dy + \int_{B_0^c\cap\{y\colon u(y)|y|^{-s}<1\}}\frac{  |u(y)|^\lambda}{|y|^{n+sp^-}} \,dy \right)\\
&\leq C 
 g( R_j^{ s} R_0^{-s} ) R_0^{sp^+}   \int_{B_0^c}\frac{ |u(y)|^\Lambda}{|y|^{n+sp^+}} \,dy +C g( R_j^{ s}  R_0^{-s} )R_0^{s(p^+-p^-)} R_0^{sp^-}   \int_{B_0^c}\frac{ |u(y)|^\lambda}{|y|^{n+sp^-}} \,dy \\
&\leq
g( R_j^{ s} R_0^{-s} )R_0^{-s} \left( \tailpp(u;R_0)^{p^+-1} 
+ \tailpm(u;R_0)^{p^--1} \right),
\end{align*}
since $R_0<1$. Moreover, it is easy to see that
\begin{align*}
(b)&\leq 
C g(\|u\|_{L^\infty(B_0)} R_j^{ s} R_0^{-s} ) R_0^{-s}.
\end{align*}
Then, from the last two expression we get
\begin{align*}
(ii)&\leq 
C(n,s) g(R_j^s R_0^{-s}) R_0^{-s} (  \tailpp(u;R_0)^{\Lambda}) +  \tailpm(u;R_0)^{\lambda}) ) \\
&\quad + C(n,s)  g(R_j^s R_0^{-s}) R_0^{-s} \max\{\|u\|_{L^\infty(B_0)}^{\Lambda}, \|u\|_{L^\infty(B_0)}^{\lambda}\} \\
&\leq C(n,s) Q(u;R_0)  g(R_j^s R_0^{-s}) R_0^{-s},
\end{align*}
where, for some $\beta=\beta(\lambda, \Lambda)>0$ we have denoted
$$
Q(u;R_0)=  (\|u\|_{L^\infty(B_0)} + \tailpp(u;R_0) + \tailpm(u;R_0))^\beta.
$$
From the last inequality we get
\begin{equation} \label{cota.2}
g^{-1}(R_j^{s} \cdot (ii)) \leq  C(n,s,\lambda, \Lambda) Q(u;R_0) R_0^{-s}  R_j^s.
\end{equation}

Plugging  \eqref{cota.1} and \eqref{cota.2} in \eqref{tail.cota} gives that
\begin{align*}
\tailg((u-m_j)_-;R_j) &\leq C(g^{-1}(R_j^{s}\cdot (i)) + g^{-1}(R_j^{s} \cdot (ii)))\\
&\leq C(n,s,\lambda, \Lambda)  R_0^{-s} R_j^s [  S(\alpha)^\frac{1}{\Lambda} L    +  Q(u;R_0) ]
\end{align*}
since $\alpha<s<sp^+$. The power $\beta>0$ in the definition of $Q(u;R_0)$ may have changed form line to line but still positive.

A similar estimate holds for $\tailg((M_j-u)_-;R_j)$.

\medskip

{\bf Step 5.} 
By using the previous computations we bound expression \eqref{cota.osc}:
\begin{align*}
\osc_{B_{j+1}} u &\leq \left(1-\frac{\sigma}{2} \right) L R_j^\alpha + C[ (R_0^s g^{-1}(K) + \tailg((M_j-u)_-;R_j) + \tailg((u-m_j)_-;R_j))\\
&\leq
4^\alpha L R_{j+1}^\alpha   \left[  \left(1-\frac{\sigma}{2} \right)  +    C  R_0^{-s} S(\alpha)^\frac{1}{\Lambda}  \right] +  4^\alpha C R_{j+1}^\alpha [ R_0^{s-\alpha} g^{-1}(K)  +R_0^{-s} Q(u,R_0)] .
\end{align*}
We choose $\alpha\in(0,s)$ universally such that
$$
4^\alpha   \left[  \left(1-\frac{\sigma}{2} \right) +    C R_0^{-s} S(\alpha)^\frac{1}{\Lambda}  \right] \leq 1-\frac{\sigma}{4}.
$$
Then
$$
\osc_{B_{j+1}} u \leq \left( \left(1-\frac{\sigma}{4}\right) L + C[R_0^{s-\alpha}g^{-1}(K) + R_0^{-s} Q(u;R_0) ]  \right) R_{j+1}^\alpha.
$$
Now we choose 
$$
L= \frac{4}{\sigma} C (R_0^{s-\alpha}g^{-1}(K) + R_0^{-s} Q(u;R_0) )
$$ 
which implies \eqref{paso.0} as  $4^{\alpha+1}C/\sigma >2$ and gives $\osc_{B_{j+1}} u \leq L R_{j+1}^\alpha$.

\medskip

{\bf Step 6.} We may pick $m_{j+1}$, $M_{j+1}$ such that
$$
m_j\leq m_{j+1}\leq \inf_{B_{j+1}} u \leq \sup_{B_{j+1}} u \leq M_{j+1} \leq M_j, \quad M_{j+1}-m_{j+1} = L R_{j+1}^\alpha, 
$$
which completes the induction and proves the claim.

\medskip

Now fix $r\in (0,R_0)$ and find an integer $j\geq 0$ such that $R_{j+1}\leq r <R_j$, thus $R_j\leq 4r$. Hence, by the claim and the election of $L$ we have that, for some $C$, 
\begin{align*}
\osc_{B_r} u \leq \osc_{B_j} u \leq L R_j^\alpha &\leq C (R_0^{s-\alpha}g^{-1}(K)+R_0^{-s}  Q(u;R_0))   r^\alpha\\
&\leq 
C(R_0^{s-\alpha} g^{-1}(K)+ R_0^{-s}Q(u;R_0))    r^\alpha
\end{align*}
which concludes the argument.
\end{proof}

As a corollary of Theorem \ref{thm.osc} we get the interior H\"older continuity of solutions:
\begin{cor}\label{cor.hold}
Assume that the hypothesis of Theorem \ref{thm.osc} are in force. Then,
\begin{equation}\label{eq.inthold}
[u]_{C^\alpha(B_{R_0/2})}\leq C(n,s,\lambda,\Lambda) (R_0^{s-\alpha}g^{-1}(K)+  R_0^{-s}Q(u;R_0)).
\end{equation}
\end{cor}

\section{Boundary behavior}\label{sec.bdry}

In this section we discuss the boundary behavior of solutions. Throughout this section, $u_0:\R\longrightarrow\R$ will be defined by
\begin{equation}\label{eq.1dprof}
u_0(x):=x_+^s.
\end{equation}
We start by showing that $u_0$ is $(s,g)-$harmonic in the positive half line:
\begin{lem}\label{1dsol}
$u_0\in\W^{s,G}_{\text{loc}}(\R)$ and $\gfls u_0  = 0$ weakly and strongly in $\R_+$.
\end{lem}

\begin{proof}

First we show that $u_0\in\W^{s,G}_{\text{loc}}(\R)$. This follows rather straightforwardly since 
\[
G\left(\frac{|u_0(x)-u_0(y)|}{|x-y|^s}\right)\frac{1}{|x-y|} \leq C\frac{|u_0(x)-u_0(y)|^{p^+}}{|x-y|^{1+sp^+}} 
\]
and we can apply the homogeneous result \cite[Lemma 3.1]{IMS}.   

To see that $u_0$ is a solution, let $\rho\in(0,1)$ take any compact set $K\subset(\rho,\rho^{-1})$ and let $x\in K$. For any $\varepsilon>0$ consider the integral 
\[
I_\varepsilon:=\int_{|x-y|>\varepsilon} g\left(\frac{|u_0(x)-u_0(y)|}{|x-y|^s}\right)\frac{u_0(x)-u_0(y)}{|u_0(x)-u_0(y)|}\frac{dy}{|x-y|^{1+s}}.
\]
We will show that 
\begin{equation}\label{eq.unif}
I_\varepsilon\longrightarrow0\text{ uniformly as }\varepsilon\rightarrow0^+.
\end{equation}
For this, notice that for $\varepsilon<x$  it holds that $0<x-\varepsilon<x+\varepsilon<\frac{x^2}{x-\varepsilon}$ 
and split 
\[
I_\varepsilon= \int_{-\infty}^0 + \int_{x+\varepsilon}^{\frac{x^2}{x-\varepsilon}}+\int_0^{x-\varepsilon} +\int_{\frac{x^2}{x-\varepsilon}}^\infty=I_1+I_2+I_3+I_4.
\]

We will estimate each term. For this, we will use the following identities:
\begin{equation}\label{eq.g1}
\left(G\left(\frac{1}{|1-t|^s}\right)\right)'=g\left(\frac{1}{|1-t|^s}\right)\frac{s}{|1-t|^{1+s}} \quad t\neq1
\end{equation}
\begin{equation}\label{eq.g2}
\left(G\left(\frac{1-t^s}{|1-t|^s}\right)\right)'=sg\left(\frac{1-t^s}{|1-t|^s}\right)\left(\frac{1-t^{s-1}}{(1-t)^{1+s}}\right)\quad t<1.
\end{equation}

Notice that $u_0\equiv 0$ in $(-\infty,0)$ so, the change of variables $t=\frac{y}{x}$ and \eqref{eq.g1} give
\begin{align*}
I_1 &=\int_{-\infty}^0 g\left(\frac{x^s}{|x-y|^s}\right)\frac{dy}{|x-y|^{1+s}} 
    =\int_{-\infty}^0 g\left(\frac{1}{|1-\frac{y}{x}|^s}\right)\frac{x^{-1-s}dy}{|1-\frac{y}{x}|^{1+s}}\\
    &=x^{-s}\int_{-\infty}^0 g\left(\frac{1}{|1-t|^s}\right)\frac{dt}{|1-t|^{1+s}} 
    =\frac{x^{-s}}{s}\int_{-\infty}^0 \left(G\left(\frac{1}{|1-t|^s}\right)\right)'\:dt = \frac{x^{-s}}{s}G(1).
\end{align*}

To bound $I_2$ notice that using \eqref{minmax2}  we have 
\begin{align*}
g\left(\frac{x^s-y_+^s}{|x-y|^s}\right)\frac{1}{|x-y|^{1+s}}  \leq   \frac{(x^s-y_+^s)^{p^--1}}{|x-y|^{1+sp^-}}	 
\end{align*}
so, using the H\"older continuity of the power $s$, 
\begin{align*}
|I_2|  \leq C\int_{x+\varepsilon}^{\frac{x^2}{x-\varepsilon}} \frac{|x-y|^{s(p^--1)}}{|x-y|^{1+sp^-}}	\:dy 		= C\frac{x^{-s}}{s} \frac{(x^s-(x-\varepsilon)^s)}{\varepsilon^s}.
\end{align*}

Next, $I_3$ is estimated using the same change of variables as that for $I_1$:
\begin{align*}
I_3 & = \int_0^{x-\varepsilon} g\left(\frac{x^s-y^s}{|x-y|^s}\right)\frac{dy}{|x-y|^{1+s}} 
	 = x^{-(1+s)}\int_0^{x-\varepsilon} g\left(\frac{1-\left(\frac{y}{x}\right)^s}{|1-\left(\frac{y}{x}\right)|^s}\right)\frac{dy}{|1-\left(\frac{y}{x}\right)|^{1+s}} \\ 
	& = x^{-s}\int_0^{1-\frac{\varepsilon}{x}} g\left(\frac{1-t^s}{|1-t|^s}\right)\frac{dt}{|1-t|^{1+s}}.
\end{align*}
Similarly, recalling that $g$ is odd and making one further change of variables (for simplicity of notation we do not change the name of the variable),
\begin{align*}
I_4 & = \int_{\frac{x^2}{x-\varepsilon}}^\infty g\left(\frac{x^s-y^s}{|x-y|^s}\right)\frac{dy}{|x-y|^{1+s}}  = -x^{-(1+s)}\int_{\frac{x^2}{x-\varepsilon}}^\infty g\left(\frac{\left(\frac{y}{x}\right)^s-1}{|\frac{y}{x}-1|^s}\right)\frac{dy}{|\frac{y}{x}-1|^{1+s}} \\ 
	& = -x^{-s}\int_{(1-\frac{\varepsilon}{x})^{-1}}^\infty g\left(\frac{t^s-1}{|t-1|^s}\right)\frac{dt}{|t-1|^{1+s}} 
 = -x^{-s}\int_0^{1-\frac{\varepsilon}{x}} g\left(\frac{t^{-s}-1}{|t^{-1}-1|^s}\right)\frac{dt}{t^2|t^{-1}-1|^{1+s}} \\ 
	& = -x^{-s}\int_0^{1-\frac{\varepsilon}{x}} g\left(\frac{t^{-s}-1}{|t^{-1}-1|^s}\right)\frac{t^{s-1}dt}{|1-t|^{1+s}} 
\end{align*}
therefore, using \eqref{eq.g2}
\begin{align*}
I_3 + I_4 & = x^{-s}\int_0^{1-\frac{\varepsilon}{x}} \left[g\left(\frac{1-t^s}{|1-t|^s}\right) - t^{s-1}g\left(\frac{t^{-s}-1}{|t^{-1}-1|^s}\right)\right]\frac{dt}{|1-t|^{1+s}}\\  
         & = x^{-s}\int_0^{1-\frac{\varepsilon}{x}} g\left(\frac{1-t^s}{|1-t|^s}\right)\frac{(1-t^{s-1})}{|1-t|^{1+s}}\:dt
 = \frac{x^{-s}}{s}\int_0^{1-\frac{\varepsilon}{x}} \left(G\left(\frac{1-t^s}{|1-t|^s}\right)\right)'\:dt\\ 
         & = \frac{x^{-s}}{s}\left(G\left(\frac{x^s-(x-\varepsilon)^s}{\varepsilon^s}\right)-G(1)\right). 
\end{align*}

Putting all the estimates together we get
\begin{equation}\label{eq.iep}
|I_\varepsilon|\leq \frac{x^{-s}}{s}\left(C\frac{(x^s-(x-\varepsilon)^s)}{\varepsilon^s}+G\left(\frac{x^s-(x-\varepsilon)^s}{\varepsilon^s}\right)\right)
\end{equation}
and hence the desired convergence. In particular, $u_0$ is a strong (and thanks to Corollary \ref{cor.strongweak} a weak) solution.  
\end{proof}

Next, we want to study the one-dimensional profile $u(x)=u_0(x_n)$ in the half space. Here $GL(n)$ stands for the general linear group  of all invertible $n\times n$ matrices.  

\begin{lem}\label{1dprof}
Let $A\in GL(n)$ and define, for $\varepsilon>0$ and $x\in \R^n_+:=\{x_n>0\}$
\[
h_\varepsilon(x,A):=\int_{B_\varepsilon^c}g\left(\frac{u_0(x_n)-u_0(z+x_n)}{|Az|^s}\right)\:\frac{dz}{|Az|^{n+s}}
\] 
and $u(x)=u_0(x_n)$. 

Then $h_\varepsilon\rightarrow0^+$ uniformly in any compact $K\subset\R^n_+\times GL(n)$. As a consequence, $u\in \widetilde{W}^{s,G}(\R^n)$ and $\gfls u=0$ weakly and strongly in $\R^n_+$.
\end{lem}

\begin{proof}

Let $A\in GL(n)$ and $K=H\times H'$ a compact subset of $\R^n_+\times GL(n)$. By the singular value decomposition we have that $AS^{n-1}$ is an ellipse with diameter bounded by the spectral norm of $A$ (here $S^{n-1}=\partial B_1$). The elliptical coordinates are given for any $y\in\R^n\setminus\{0\}$ by
\[
y=\rho\omega,\quad \rho>0, \quad \omega\in AS^{n-1}.
\]  
Then $dy=\rho^{n-1}d\rho d\omega$ with $d\omega$ the surface element of $AS^{n-1}$. Let us further call
\[
e_A:=A^{-1}e_n,\quad E_A:=\{x\in\R^n:x\cdot e_A>0\}.
\]

Now, with the change of variables $Az=y$ we compute
\begin{align*}
h_\varepsilon(x,A) & =\int_{B_\varepsilon^c}g\left(\frac{u(x)-u(x+z)}{|Az|^s}\right)\:\frac{dz}{|Az|^{n+s}} \\
				   & = \int_{(AB_\varepsilon)^c}g\left(\frac{u(x)-u(x+A^{-1}y)}{|y|^s}\right)\:\frac{dy}{|\text{det}A||y|^{n+s}} \\
				   & = \int_{AS^{n-1}}\frac{1}{|\text{det}A||\omega|^{n+s}}\int_\varepsilon^\infty g\left(\frac{u_0(x_n)-u_0(x_n+\omega\cdot e_A\rho)}{| \omega\rho|^s}\right)\:\frac{d\rho}{\rho^{1+s}}d\omega \\
				   & = \int_{AS^{n-1}\cap E_A}\frac{1}{|\text{det}A||\omega|^{n+s}}\int_{(-\varepsilon,\varepsilon)^c} g\left(\frac{u_0(x_n)-u_0(x_n+\omega\cdot e_A\rho)}{| \omega\rho|^s}\right)\:\frac{d\rho}{|\rho|^{1+s}}d\omega \\
				   & = \int_{AS^{n-1}\cap E_A}\frac{|\omega\cdot e_A|^{1+s}}{|\text{det}A||\omega|^{n+s}}\int_{(-\varepsilon,\varepsilon)^c} g\left(\sigma(\omega)\frac{u_0(x_n)-u_0(x_n+\omega\cdot e_A\rho)}{|\omega\cdot e_A\rho|^s}\right)\:\frac{d\rho}{|\omega\cdot e_A\rho|^{1+s}}d\omega
\end{align*}
where 
\[
\sigma(\omega):=\left(\frac{\omega}{|\omega|}\cdot e_A\right)^s.
\]
 
Notice that thanks to the bound \eqref{eq.iep} we have, for $\varepsilon$ small enough depending on the norm of $A$,
\begin{align*}
\int_{(-\varepsilon,\varepsilon)^c} &g\left(\frac{u_0(x_n)-u_0(x_n+\omega\cdot e_A\rho)}{|\omega\cdot e_A\rho|^s}\right)\:\frac{d(\omega\cdot e_A\rho)}{|\omega\cdot e_A\rho|^{1+s}}\\  & \leq \frac{x_n^{-s}}{s}\left(C \left(\frac{x_n^s-(x_n-\varepsilon)^s}{\varepsilon^s}\right)  
							  + G\left(\frac{x_n^s-(x_n-\varepsilon)^s}{\varepsilon^s}\right)\right)  =: Cx_n^{-s}\psi(x_n,\varepsilon).
\end{align*}

Therefore, using \eqref{minmax1} we get
\[
|h_\varepsilon(x,A)| \leq Cx_n^{-s}\int_{AS^{n-1}\cap E_A}\frac{|\omega\cdot e_A|^{1+s}\max\{\sigma(\omega)^{sp^-},\sigma(\omega)^{sp^+}\}}{|\text{det}A||\omega|^{n+s}}\psi(x_n,\omega\cdot e_A\varepsilon)d\omega.
\]

Now for $\varepsilon>0$
\[
\frac{\partial}{\partial\varepsilon}\psi(x_n,\varepsilon)=\frac{s}{\varepsilon^{s+1}}\left(\varepsilon(x_n-\varepsilon)^{s-1}-x_n^s+(x_n-\varepsilon)^s\right)\geq0. 
\]
Also, changing back the variables we have
\begin{align*}
\int_{AS^{n-1}}&\frac{|\omega\cdot e_A|^{1+s}\max\{\sigma(\omega)^{sp^-},\sigma(\omega)^{sp^+}\}}{|\text{det}A||\omega|^{n+s}}d\omega\\
 & = \int_{S^{n-1}}\frac{|\omega\cdot e_n|^{1+s}\max\left\{\left(\frac{\omega\cdot e_n}{|A\omega|}\right)^{sp^-},\left(\frac{\omega\cdot e_n}{|A\omega|}\right)^{sp^+}\right\}}{|A\omega|^{n+s}}d\omega \\
 & \leq \mathcal{H}^{n-1}(S^{n-1})\|A^{-1}\|^{n+s}\max\left\{\|A^{-1}\|^{sp^-},\|A^{-1}\|^{sp^+}\right\}
\end{align*}
where $\mathcal{H}^{n-1}$ is the $(n-1)-$dimensional Hausdorff measure and we used that 
\[
|\omega\cdot e_A|\leq |\omega||A^{-1}e_n|\leq \|A\|\|A^{-1}\|.
\]

Therefore, we get 
\[
|h_\varepsilon(x,A)| \leq Cx_n^{-s}\psi(x_n,\|A\|\|A^{-1}\|\varepsilon)
\]
and the result follows by taking $\varepsilon\rightarrow0$. From this result, the conclusions on $u(x)=u_0(x_n)$ follow straightforwardly.
\end{proof}

Once the $1-$d profile is known to be a solution, we wish to straighten the boundary of $\Omega$. The following lemma asserts that when we do that the fractional $g-$Laplacian of the profile remains bounded: 

\begin{lem}\label{lem.difeo}
Let $\Phi$ be a $C^{1,1}$ diffeomorphism in $\R^n$ such that $\Phi=Id$ in $B_r^c$ for some $r>0$ and define
\begin{equation}\label{eq.vdifeo}
v(x):=(\Phi^{-1}(x)\cdot e_n)_+^s.
\end{equation}

Then $v\in\W^{s,G}_{\text{loc}}(\R^n)$ and $\gfls v=f$ weakly in $\Phi(\R^n_+)$ with
\begin{equation}\label{eq.boundf}
\|f\|_\infty\leq  C\left(\|D\Phi\|_\infty, \|D\Phi^{-1}\|_\infty,r\right)\|D^2\Phi\|_\infty.
\end{equation}
\end{lem}

\begin{proof}

We want to show that 
\begin{equation}\label{eq.h}
h_\varepsilon(x):=\int_{\{|\Phi^{-1}(x)-\Phi^{-1}(y)|>\varepsilon\}}g\left(\frac{v(x)-v(y)}{|x-y|^s}\right)\frac{dy}{|x-y|^{n+s}}
\end{equation}
converges in $L^1(K)$ for any compact set $K\subset\R^n_+$ to some $f\in L^\infty(\R^n_+)$ satisfying \eqref{eq.boundf}. Once this is proven Lemma \ref{lem.lim} gives the result. 

Let us make the change of variables $\Phi(\bar{x})=x$, denote $J(\cdot)=|\text{det}D\Phi(\cdot)|$ and write \eqref{eq.h} as
\begin{align*}
h_\varepsilon(x) =& \int_{B_\varepsilon^c(\bar{x})}g\left(\frac{v(\Phi(\bar{x}))-v(\Phi(\bar{y}))}{|\Phi(\bar{x})-\Phi(\bar{y})|^s}\right)\frac{J(\bar{y})d\bar{y}}{|\Phi(\bar{x})-\Phi(\bar{y})|^{n+s}} \nonumber \\		
				= & \int_{B_\varepsilon^c(\bar{x})}g\left(\frac{u_0(\bar x)-u_0(\bar y)}{|\Phi(\bar{x})-\Phi(\bar{y})|^s}\right)\frac{\zeta(\bar{x},\bar{y})d\bar{y}}{|D\Phi(\bar{x})(\bar{x}-\bar{y})|^{n+s}} \nonumber \\
				 & +\int_{B_\varepsilon^c(\bar{x})}g\left(\frac{u_0(\bar x)-u_0(\bar y)}{|\Phi(\bar{x})-\Phi(\bar{y})|^s}\right)\frac{J(\bar{x})d\bar{y}}{|D\Phi(\bar{x})(\bar{x}-\bar{y})|^{n+s}} 
				 :=  I_1+I_2
\end{align*}
with
\[	
\zeta(\bar{x},\bar{y}):=\frac{|D\Phi(\bar{x})(\bar{x}-\bar{y})|^{n+s}}{|\Phi(\bar{x})-\Phi(\bar{y})|^{n+s}}J(\bar{y})-J(\bar{x}).
\]

Now, we use ellipticity and the fact that 
\[
\frac{|D\Phi(\bar{x})(\bar{x}-\bar{y})|^{s\Lambda}}{|\Phi(\bar{x})-\Phi(\bar{y})|^{s\Lambda}}\leq C(\|D\Phi\|_\infty,\|D\Phi^{-1}\|_\infty).
\]
to get
\begin{align*}
|I_1| & \leq \int_{B_\varepsilon^c(\bar{x})}g\left(\frac{u_0(\bar x)-u_0(\bar y)}{|\Phi(\bar{x})-\Phi(\bar{y})|^s}\right)\frac{|\zeta(\bar{x},\bar{y})|d\bar{y}}{|D\Phi(\bar{x})(\bar{x}-\bar{y})|^{n+s}}  \\
	& \leq \int_{B_\varepsilon^c(\bar{x})}\frac{(u_0(\bar x)-u_0(\bar y))^\Lambda}{|\Phi(\bar{x})-\Phi(\bar{y})|^{s\Lambda}}\frac{|\zeta(\bar{x},\bar{y})|d\bar{y}}{|D\Phi(\bar{x})(\bar{x}-\bar{y})|^{n+s}} \\
	& \leq \int_{B_\varepsilon^c(\bar{x})}\frac{(u_0(\bar x)-u_0(\bar y))^\Lambda}{|D\Phi(\bar{x})(\bar{x}-\bar{y})|^{n+sp^+}}|\zeta(\bar{x},\bar{y})|d\bar{y}. 
\end{align*}
Therefore $I$ can be bound exactly as the second term of (3.6) in \cite{IMS}.

$I_2$ on the other hand vanishes identically as $\varepsilon\rightarrow0^+$ by means of Lemma \ref{1dprof}; indeed, since $D\Phi(\R^n)$ is a compact subset of $GL(n)$ the integral vanishes uniformly in any compact set $\Phi^{-1}(K)\subset\R^n_+$, therefore in any compact set $K\subset\Phi(\R^n_+)$. 
\end{proof}

We will need the following Lemma regarding the geometric properties of a $C^{1,1}$ domain (cf. \cite{HKS}). It essentially says that any point on $\partial\Omega$ has an interior and an exterior tangent ball and the distance function behaves like $|\cdot|$ close to such a point inside $\Omega$. Recall that the distance function to $\partial\Omega$ is given by $d(x):=\text{dist}(x,\Omega^c)$. 

\begin{lem}\label{lem.dist}
Let $\Omega$ be an bounded domain in $\R^n$ with $C^{1,1}$ boundary. Then, there exists $\rho>0$ such that for any $x_0\in\partial \Omega$ there exist $x_1,x_2\in\R^n$ in the normal line to $\partial\Omega$ at $x_0$ such that
\begin{enumerate}

\item $B_\rho(x_1)\subset\Omega$ and $B_\rho(x_2)\subset\Omega^c$;

\item $\overline B_\rho(x_1)\cap\overline B_\rho(x_2)=\{x_0\}$;

\item $d(x)=|x-x_0|$ for any $x=(1-t)x_0+tx_1\:,t\in[0,1]$.

\end{enumerate}
\end{lem}

We are in position to show that $\gfls d^s$ is bounded in a neighborhood of $\partial\Omega$:
\begin{prop}\label{prop.distance}
Let $\Omega$ be a bounded domain in of $\R^n$ with $C^{1,1}$ boundary. Then there exists $\rho>0$ such that 
\[
\gfls d^s=f\quad\text{ weakly in }\Omega_\rho
\]
for some $f\in L^\infty(\Omega_\rho)$ with $\Omega_\rho:=\{x\in\Omega:d(x)<\rho\}$.
\end{prop}

\begin{proof}

By taking a finite covering of $\Omega_\rho$ by balls centered at points in $\partial\Omega$ and a partition of unity, it is enough to show that $\gfls d^s=f$ holds weakly in $\Omega\cap B_{2\rho}$ with $\rho$ small enough, depending only in the geometry of $\Omega$. To that aim, we are going to flatten the boundary of $\Omega$ near the origin: let $\Phi(\bar{x})=x$ be a $C^{1,1}$ diffeomorphism such that $\Phi=Id$ in $B_{4\rho}^c$ such that
\[
\Omega\cap B_{2\rho}\subset\subset\Phi(B_{3\rho}\cap\R^n_+),\quad d(\Phi(\bar{x}))=(\bar{x}_n)_+\text{ for }\bar{x}\in B_{3\rho}.
\]

We will show that 
\[
h_\varepsilon(x):=\int_{\{|\Phi^{-1}(x)-\Phi^{-1}(y)|>\varepsilon\}}g\left(\frac{d^s(x)-d^s(y)}{|x-y|^s}\right)\frac{dy}{|x-y|^{n+s}}\longrightarrow f\text{ in } L^1_{\text{loc}}(\Omega\cap B_{2\rho}) 
\]
for a function $f\in L^\infty$ and the result will follow from Lemma \ref{lem.lim}. 

Setting $\bar{x}=\Phi^{-1}(x)$ and changing variables we can compute (with $J$ defined as in Lemma \ref{lem.difeo}) we can write
\begin{align*}
h_\varepsilon(x) & =\int_{B_\varepsilon^c(x)}g\left(\frac{d^s(\Phi(\bar{x}))-d^s(\Phi(\bar{y}))}{|\Phi(\bar{x})-\Phi(\bar{y})|^s}\right)\frac{J(\bar{y})d\bar{y}}{|\Phi(\bar{x})-\Phi(\bar{y})|^{n+s}}\\
				 & = \int_{B_\varepsilon^c(x)\cap B_{3\rho}}g\left(\frac{d^s(\Phi(\bar{x}))-d^s(\Phi(\bar{y}))}{|\Phi(\bar{x})-\Phi(\bar{y})|^s}\right)\frac{J(\bar{y})d\bar{y}}{|\Phi(\bar{x})-\Phi(\bar{y})|^{n+s}}\\
				 & +\int_{B_{3\rho}^c}g\left(\frac{d^s(\Phi(\bar{x}))-d^s(\Phi(\bar{y}))}{|\Phi(\bar{x})-\Phi(\bar{y})|^s}\right)\frac{J(\bar{y})d\bar{y}}{|\Phi(\bar{x})-\Phi(\bar{y})|^{n+s}} \\
				 & =\int_{B_\varepsilon^c(\bar{x})}g\left(\frac{u_0(\bar{x}_n)-u_0(\bar{y}_n)}{|\Phi(\bar{x})-\Phi(\bar{y})|^s}\right)\frac{J(\bar{x})d\bar{y}}{|\Phi(\bar{x})(\bar{x}-\bar{y})|^{n+s}} \\
				 & +\int_{B_{3\rho}^c}\left( g\left(\frac{d^s(\Phi(\bar{x}))-d^s(\Phi(\bar{y}))}{|\Phi(\bar{x})-\Phi(\bar{y})|^s}\right)-g\left(\frac{u_0(\bar{x}_n)-u_0(\bar{y}_n)}{|\Phi(\bar{x})-\Phi(\bar{y})|^s}\right)\right)\frac{J(\bar{y})d\bar{y}}{|\Phi(\bar{x})-\Phi(\bar{y})|^{n+s}} \\
				 & =f_{1,\varepsilon}(\bar{x})+f_2(\bar{x}).
\end{align*}
 
As in Lemma \ref{lem.difeo} (and using Lemma \ref{1dprof}) for $f_1 \in L^\infty_{\text{loc}}(\R^n_+)$ we have that 
\[
\lim_{\varepsilon\rightarrow0^+}f_{1,\varepsilon}=f_1\text{ in }L^1_{\text{loc}}(\Omega\cap B_{2\rho}).
\] 

It remains to bound $f_2$. To do that, we note that 
\[
\text{dist}(\Phi^{-1}(\Omega\cap B_{2\rho}),\Phi(B_{3\rho}^c)\geq \theta>0
\]
for some $\theta$ depending only on $\rho$ and $\Phi$. Now using that $d^s\circ\Phi$ is $s$-H\"older continuous (and so is $u_0$) and the properties of $\Phi$ we have
\[
\left|g\left(\frac{d^s(\Phi(\bar{x}))-d^s(\Phi(\bar{y}))}{|\Phi(\bar{x})-\Phi(\bar{y})|^s}\right)-g\left(\frac{u_0(\bar{x}_n)-u_0(\bar{y}_n)}{|\Phi(\bar{x})-\Phi(\bar{y})|^s}\right)\right| \leq C 
\]
so that 
\[
|f_2(\bar{x})|\leq C\int_{B_{3\rho}^c}\frac{d\bar{y}}{|\Phi(\bar{x})-\Phi(\bar{y})|^{n+s}} \leq C
\]
and we get the result.
\end{proof}

Next, we want to construct the appropriate barriers to get bounds on our solutions in terms if the distance function $d^s$. We start by considering functions whose fractional $g-$Laplacian is constant in the unit ball.

\begin{lem}\label{lem.sol1}
The equation
\begin{equation}\label{eq.sol1}
\left\{
\begin{array}{rcl}
 \gfls v &=  & 1 \text{ in }B_1 \\
  v &= &0 \text{ in }B_1^c
\end{array}
\right.
\end{equation}
has a unique solution $v_0\in W^{s,G}_0(\Omega)$. Moreover, $v_0\in L^\infty(\R^n)$, is radially symmetric, nonincreasing and for any $r\in(0,1)$ it holds that $\inf_{B_r}v_0>0$.
\end{lem}

\begin{proof}
First, weak solutions of \eqref{eq.sol1} are constructed as minimizers in $W^{s,G}_0(B_1)$ of  
\[
J(v):=\iint G\left(\frac{v(x)-v(y)}{|x-y|^s}\right)\frac{dxdy}{|x-y|^n}-\int_{B_1}v\:dx
\]
so existence and uniqueness follow from the direct method of the Calculus of Variations. Thanks to the rotational invariance of the equation given in Lemma \ref{lem.rotation} we also have $v_0(x)=\psi(|x|)$ for some $\psi:\R_+\longrightarrow\R_+$. Further, by the P\'olya-Szeg\"o principle proved in \cite{DNFBS} we have that $\psi$ is nonincreasing. 

Now let
\[
r_0:=\inf\{r\in\R_+:\psi(r)=0\}
\]
and let us show the last assertion by proving that $r_0=1$. It is clear that $r_0\in (0,1]$, as $\psi$ vanishes for $r>1$. Let us then assume by contradiction that $r_0\in(0,1)$. Then, 
\[
\left\{
\begin{array}{rcl}
 \gfls v_0 &=  & 1 \text{ in }B_{r_0} \\
  v_0 &= &0 \text{ in }B_{r_0}^c.
\end{array}
\right.
\]

Next denote $\tilde{v}_0(x):=v_0(r_0x)$ and notice that Lemma \ref{lem.scaling} gives that
\[
\left\{
\begin{array}{rcl}
 (-\Delta_{g_{r_0}})^s \tilde{v}_0 &=  & 1 \text{ in }B_1 \\
  \tilde{v}_0 &= &0 \text{ in }B_1^c
\end{array}
\right.
\]
which implies 
\[
\left\{
\begin{array}{rcl}
 (-\Delta_{g})^s \tilde{v}_0 &\leq & r_0^s<1\text{ in }B_1 \\
  \tilde{v}_0 &= &0 \text{ in }B_1^c
\end{array}
\right.
\]
and the comparison principle implies that $v_0(x)\geq v_0(r_0x)$, or $\psi(r)\geq\psi(r_0r)$ for any $r\in(0,r_0)$. In particular, 
\[
0\leq \psi(r_0^2)\leq\psi(r_0)=0
\]
so that $\psi(r_0^2)=0$ which is a contradiction with the definition of $r_0$. 

It remains to show that $v_0\in L^\infty(\R^n)$. Let
\[
w(x):=\min\{(2-x_n)_+^s,5^s\}\in C^s(\R^n)\cap\W^{s,G}(B_1) 
\]
and notice that, for $x\in B_2,\:w(x)=u_0(2-x_n)$ with $u_0$ as defined in \eqref{eq.1dprof}. Then, we can apply Lemma \ref{nonlocal.behavior} in $B_{3/2}$ with 
\[
u(x)=u_0(2-x_n),\:f\equiv0\text{ and }v(x)= (u_0(2-x_n)-5^s)_+
\]
to get, using Lemma \ref{1dprof},
\[
\gfls w(x)=2\int_{\{y_n\leq -3\}}\left[g\left(\frac{(2-x_n)^s_+-5^s}{|x-y|^s} \right) - g\left(\frac{(2-x_n)^s_+ -(2-y_n)^s_+}{|x-y|^s} \right) \right]\frac{ dy}{|x-y|^{n+s}} 
\]
weakly in $B_1$. The right hand side of this expression is a positive continuous function of $x$ and hence bounded below in $B_1$ by some positive constant $\eta$, i.e. $\gfls w\geq\eta>0$ weakly in $B_1$. Now, choose $c>0$ such  that $\min\{c^\lambda,c^\Lambda\}=\eta^{-1}$ and use \eqref{minmax1} to get $\gfls (c w)\geq1$,  which means that $\gfls (c w)\geq \gfls v_0$.  This, together with the fact that $v_0=0\leq c w$ in $B_1^c$
gives, through Proposition \ref{compara}, $0\leq v_0\leq c w$ in $\R^n$ and hence $0\leq v_0\leq \frac{5^s}{c}$ in $\R^n$ as desired.
\end{proof}

As a consequence of the previous lemma we obtain that function with bounded fractional $g-$Laplacian are themselves bounded.

\begin{prop}\label{prop.bound}
Let $u\in W^{s,G}_0(\Omega)$ be a weak solution of $|\gfls u|\leq K$ in $\Omega$ for some $K>0$. Then
\begin{equation}\label{eq.bound}
\|u\|_{L^\infty(B_1)}\leq C
\end{equation}
where $C$ is a positive constant depending only on $s,n,\lambda,\Lambda, K$ and diam$(\Omega)$. 
\end{prop}

\begin{proof}
Let $d>$diam$(\Omega)$ and take $x_0\in\Omega$ such that $\Omega\subset\subset B_d(x_0)$. Consider $v_0$ as in the previous Lemma and notice that thanks the translation invariance, the scaling from Lemma \ref{lem.scaling} and \eqref{minmax1}
\[
\gfls v_0\left(\frac{x-x_0}{d}\right)\geq \frac{1}{\max\{d^{s\lambda},d^{s\Lambda}\}}\quad\text{ weakly in }B_d(x_0).
\] 
As in the previous theorem, multiplying $v_0$ by a constant $C$ (which will depend only on universal parameters and $d$)
\[
\gfls C v_0\left(\frac{x-x_0}{d}\right)\geq K\quad\text{ weakly in }\Omega
\] 
and since $u=0\leq Cv_0$ in $\Omega^c$, the comparison principle gives $u\leq Cv_0$ in $\R^n$. We analogously bound $-u$ to get \eqref{eq.bound} and the proof concludes.   
\end{proof}

In the next lemma we construct the barrier that we need to compare $u$ with $d^s$:
\begin{lem}\label{lem.barrier}
There exist $w\in C^s(\R^n)$, $R>0$, $\eta\in (0,1)$ and $c>1$ such that 
\[
\gfls w\geq \eta\quad\text{ weakly in }B_R(e_n)\setminus\overline{B_1}
\]
and 
\[
c^{-1}(|x|-1)^s_+\leq w\leq c(|x|-1)^s_+\quad\text{ in }\R^n.
\]
\end{lem}

\begin{proof}

Since the fractional $g-$Laplacian is translation invariant (and also rotation invariant, recall Lemma \ref{lem.rotation}) by using a similar scaling argument as the one of Lemma \ref{lem.sol1} it suffices to prove the result for any ball of radius $R>2$ and any point $\bar{x}_R$ on its boundary. Let us set $\tilde{x}_R:=(0,-(R^2-4)^{1/2})$ and $\bar{x}_R=\tilde{x}_R+Re_n$. In this way, $B_R(\bar{x}_R)$ intersects the hyperplane $\{x_n=0\}$ at the $n-1$ dimensional ball $\{|x'|<2\}$ where we denote as usual $(x',x_n)\in \R^{n-1}\times\R$. 

For $R>2$ there exists $\varphi\in C^{1,1}(\R^{n-1})$ with $\|\varphi\|_{C^{1,1}(\R^{n-1})}\leq C/R$ and  
\[
\varphi(x')=((R^2-|x'|)^{1/2}-(R^2-4)^{1/2})_+,\quad\text{ for all }|x'|\in[0,1]\cup[3,\infty)
\]
and set $U_+:=\{x\in \R^{n-1}:\varphi(x')<x_n\}$.

Further, by the same construction as in \cite[Lemma 4.3]{IMS} we have a diffeomorphism $\Phi\in C^{1,1}(\R^n,\R^n)$ such that $\Phi(0)=\bar{x}_R$, $\Phi=Id$ in $B_4^c$, 
\[
\|\Phi-Id\|_{C^{1,1}(\R^n,\R^n)}+\|\Phi^{-1}-Id\|_{C^{1,1}(\R^n,\R^n)}\leq \frac{C}{R},\quad\Phi(\R^n_+)=U_+.
\]

Next, let $v$ be defined as in \eqref{eq.vdifeo}, so that Lemma \ref{lem.difeo} gives that 
\[
\gfls v=f\text{ weakly in }U_+\text{ and }\|f\|_\infty\leq \frac{C}{R}.
\]
and note that, by the properties of $\Phi$, $v(x)=u_0(x_n)\text{ in }B_4^c$, $v\leq 4^s$ in $B_4$.

Let us further truncate $\hat{v}:=\min\{v,5^s\}$ and observe that 
\[
v(x)-\hat{v}(x)=(x_n)_+^s-5^s\text{ in }\{x_n\geq 5\},\quad v-\hat{v}=0\text{ in }\{x_n<5\}
\]
and in particular $v-\hat{v}$ vanishes identically in $B_4$ so that, using Lemma \ref{nonlocal.behavior} 
\[
\gfls \hat{v}=\gfls(v+(\hat{v}-v))=f+h\text{ weakly in }B_4
\] 
with
\begin{align*}
h(x) & =2\int_{B_4^c} \left[g\left(\frac{v(x) -\hat{v}(y)}{|x-y|^s} \right) - g\left(\frac{v(x) -v(y)}{|x-y|^s} \right) \right]\frac{ 1}{|x-y|^{n+s}} \,dy \\
	 & \geq 2\int_{\{y_n\geq5\}} \left[g\left(\frac{(x_n)_+^s -5^s}{|x-y|^s} \right) - g\left(\frac{(x_n)_+^s -(y_n)_+^s}{|x-y|^s} \right) \right]\frac{ 1}{|x-y|^{n+s}} \,dy
\end{align*}
for $x\in B_4$. Choosing an appropriate constant $\eta$ as in the proof of Lemma \ref{lem.sol1},
\[
\gfls \hat{v}= f+g\geq -\frac{C}{R}+\eta\text{ weakly in }U_+\cap B_4
\]
and taking $R$ large enough we get that 
\begin{equation}\label{eq.hatvsupersol}
\gfls \hat{v}\geq \frac{\eta}{2}\text{ weakly in }U_+\cap B_2(\bar{x}_R).
\end{equation}

We are ready to estimate $\hat{v}$. To do that, let us define $d_R(x):=(|x-\tilde{x}_R|-R)_+$ and notice that we can immediately find $\tilde{c}>1$ such that 
\begin{equation}\label{eq.boundabovev}
\hat{v}(x)\leq \tilde{c}d^s_R(x)\quad \text{ for any }x\in\R^n.
\end{equation}
In fact, no equation is used here since $\hat{v}$ vanishes in $U_+^c\supset B_R(\tilde{x}_R)$ and it is $s-$H\"older continuous in $\R^n$.

To get the lower bound, notice that if $x\in B_1(\bar{x}_R)$ then either $$x\in B_1(\bar{x}_R)\cap U_+^c\subset B_R(\tilde{x}_R)$$ and $d_R^s(x)=\tilde{c}\hat{v}(x)=0$ or 
$$x\in B_1(\bar{x}_R)\cap U_+\subset B_R^c(\tilde{x}_R).$$ In the latter case, we can let $(X',X_n)$ be such that $x=\Phi(X)$, $Z=(X',0)$ and $z=\Phi(Z)$. Then $|X'|<1$ and $z\in\partial B_R(\tilde{x}_R)$ so that 
\[
d_R^s(x)\leq |x-z|^s\leq \tilde{c}|X-Z|^s=\tilde{c}X_n^s=\tilde{c}\hat{v}(x)
\]
so, taking $\tilde{c}>1$ bigger if needed
\begin{equation}\label{eq.lboundbelowv}
\hat{v}\geq \frac{1}{\tilde{c}}d^s_R\quad \text{ in }B_R(\tilde{x}_R).
\end{equation}

We want to extend \eqref{eq.lboundbelowv} but keeping \eqref{eq.hatvsupersol} and \eqref{eq.boundabovev}. Take $\varepsilon\in(0,1/\tilde{c})$ and 
\[
v_\varepsilon:=\max\{\hat{v},\varepsilon d^s_R\}.
\] 
We have that $v_\varepsilon$ satisfies the corresponding estimates \eqref{eq.boundabovev} and \eqref{eq.lboundbelowv} with a constant $c_\varepsilon=\max\{\tilde{c}+\varepsilon,\varepsilon^{-1}\}$ and further
\[
\hat{v}\leq v_\varepsilon\leq \hat{v}+\varepsilon d^s_R\text{ in }\R^n,\quad v_\varepsilon-\hat{v}=0\text{ in }B_1(\bar{x}_R).
\]  
Therefore, using again Lemma \ref{nonlocal.behavior} together with \eqref{eq.hatvsupersol} and Lemma \ref{lema.0bis} we have
\begin{align*}
\gfls v_\varepsilon & =\gfls \hat{v}-2\int_{B_{1/2}^c(\bar{x}_R)} \left[g\left(\frac{\hat{v}(x) -\hat{v}(y)}{|x-y|^s} \right) - g\left(\frac{\hat{v}(x) -v_\varepsilon(y)}{|x-y|^s} \right) \right]\frac{ 1}{|x-y|^{n+s}} \,dy \\
					& \geq \frac{\eta}{2}-C\int_{B_1^c(\bar{x}_R)} \frac{\max\{\varepsilon d^s_R(x),g(\varepsilon d^s_R(x))\}}{|x-y|^{n+2s}} \,dy. 
\end{align*}
Noticing that the second term is finite and vanishes as $\varepsilon\rightarrow0^+$ independently of $x$, we can choose $\varepsilon$ small enough so that 
\[
\gfls v_\varepsilon \geq \frac{\eta}{4}\quad\text{ weakly in } B_{1/2}(\bar{x}_R)\setminus B_R(\tilde{x}_R).
\] 

Finally, up to proceeding as in the proof of Lemma \ref{lem.sol1} if needed, the function $w(x):=v_\varepsilon(\tilde{x}_R+Rx)$ fulfills the desired properties.
\end{proof}

Now we can prove the main result of this section:
\begin{thm}\label{thm.bdry}
Let $u\in W^{s,G}_0(\Omega)$ be a weak solution of $|\gfls u|\leq K$ in $\Omega$ for some $K>0$. Then
\begin{equation}\label{eq.bdry}
|u|\leq Cd^s\quad\text{ a.e. in }\Omega
\end{equation}
where $C$ is a positive constant depending only on $s,n,\lambda,\Lambda, K, g$ and $\Omega$. 
\end{thm}

\begin{proof}

Thanks to Proposition \ref{prop.bound}, and by taking a larger constant $C$ if needed, it is enough to show \eqref{eq.bdry} in a neighborhood of $\partial\Omega$. Let
\[
U:=\left\{x\in\Omega:d(x)<\frac{R\rho}{2}\right\}
\]
where $R$ is given in Lemma \ref{lem.barrier} and $\rho$ is given in Lemma \ref{lem.dist}. Let $\bar{x}\in U$ and $x_0\in\partial\Omega$ at minimal distance from $\bar{x}$. 

According to the referred lemmata, there exist two balls $B_{\rho/2}(x_1)$ and $B_{\rho}(x_2)$ which are tangent to $\partial\Omega$ and a function in $C^s(\R^n)$ such that 
\begin{equation}\label{eq.equiationw}
\gfls w\geq \eta\quad\text{ weakly in }B_{R\rho/2}(x_0)\setminus B_{\rho/2}(x_1)
\end{equation}
\begin{equation}\label{eq.boundw}
c^{-1}\delta^s\leq w\leq c\delta^s\quad\text{ in }\R^n
\end{equation}
where $\delta(x)=$dist$(x,B_{\rho/2}^c(x_1))$. Recall that from Lemma \ref{lem.dist} we also have
\begin{equation}\label{eq.deltad}
\delta(\bar{x})=d(\bar{x})=|\bar{x}-x_0|
\end{equation}
and that further 
\[
\delta(x)\geq \theta>0\quad\text{ in }B_\rho^c(x_2)\setminus B_{R\rho/2}(x_0)
\]
for a constant $\theta$ depending only on $\Omega$. This inequality, together with\eqref{eq.boundw} and the fact that $\Omega\subset B_\rho^c(x_2)$ gives 
\[
w(x)\geq c^{-1}\theta^s\quad\text{ in }\Omega\setminus B_{R\rho/2}(x_0)
\]
and we may assume without loss of generality that $c^{-1}\theta^s<1$. 

We want to apply the comparison principle in the set $V:=\Omega\cap B_{R\rho/2}(x_0)$; set 
\[
M:=\frac{c}{\theta^s}g^{-1}\left(\frac{C}{\eta}\right)\quad\text{ and }\bar{w}=Mw
\]
where $C$ is the constant from \eqref{eq.bound}.

Recall again that we can increase the constants if needed and get
\[
\gfls \bar{w}\geq \gfls u \quad\text{ weakly in }V
\]
and since, by construction, $\bar{w}\geq u$ in $V^c$ the comparison principle and \eqref{eq.boundw} give
\[
u(x)\leq \bar{w}(x)\leq cM\delta^s(x)\quad\text{ a.e. in }\R^n
\] 
so recalling \eqref{eq.deltad} we have
\[
u(\bar{x})\leq cM\delta^s(\bar{x})=cMd^s(\bar{x})\quad\text{ for any }\bar{x}=x_0-t\nu_{x_0},\: t\in\left[0,\frac{R\rho}{2}\right]
\] 
where $\nu_{x_0}$ is the exterior unit normal to $\partial\Omega$ at $x_0$. A similar argument applied to $-u$ gives the other bound and the result is proven.
\end{proof}

\section{Proof of Theorem \ref{thm.main}}\label{sec.main}

In this section we give the proof of our main result:

\begin{proof}[Proof of Theorem \ref{thm.main}]
We set $K=\|f\|_{L^\infty(\Omega)}$. By Proposition \ref{prop.bound} we have that
$$
\|u\|_{L^\infty(\Omega)}\leq C
$$
where $C$ is a positive constant depending only on $s,n,\lambda,\Lambda, K$ and diam$(\Omega)$. 

Let us deal with the H\"older seminorm. Let $\alpha\in (0,s]$ be the exponent given in Corollary \ref{cor.hold}. Through a covering argument, inequality \eqref{eq.inthold} implies that $u\in C^\alpha_{loc}(\overline{\Omega'})$ for all $\Omega'$ compactly contained in $\Omega$, with a bound of the form
$$
\|u\|_{C^\alpha(\overline{\Omega'})}\leq C_{\Omega'} g^{-1}(K), \qquad C_{\Omega'}=C(n,s,\lambda,\Lambda,\Omega,\Omega').
$$
Therefore, it suffices to prove \eqref{global.bound} in the closure of a fixed $\rho-$neighborhood of $\partial\Omega$. Assume that $\rho=\rho(\Omega)>0$ is small enough  such that Lemma \ref{lem.dist} holds,  and thus the metric projection 
$$
\Pi \colon V \to \partial\Omega, \qquad \Pi(x)=\argmin_{y\in\Omega^c}\,|x-y|
$$
is well defined on $V:=\{x\in\overline\Omega\colon d(x)\leq \rho\}$. We claim that
\begin{equation}\label{eq.in1}
[u]_{C^\alpha(B_{r/2})}\leq C_\Omega  \quad \text{for all }x\in V \text{ and } r=d(x)
\end{equation}
for some constant $C_\Omega=C(n,s,\lam,\Lambda,\Omega, K)$, independent on $x\in V$. Recall that Corollary \ref{cor.hold} states that
$$
[u]_{C^\alpha(B_{r/2}(x))}\leq C (r^{s-\alpha}g^{-1}(K)+  r^{-s}  \|u\|_{L^\infty(B_r(x))}^\beta +  r^{-s} \tailpp(u;x,r)^\beta  +    r^{-s} \tailpm(u;x,r)^\beta ).
$$
where $C$ is a constant depending on $n$, $s$, $\lambda$, $\Lambda$, $\Omega$. The first term in the right hand side of the previous inequality can be bounded as
$$
g^{-1}(K) r^{s-\alpha}   \leq  g^{-1}(K) \rho^{s-\alpha} \leq C(K,\lambda,\Lambda) \rho^{s-\alpha}.
$$
For the second one we use Theorem \ref{thm.bdry} and the fact that $\alpha\leq s$ to obtain 
$$
\|u\|_{L^\infty(B_r(x))}\leq C (d(x)+r)^s \leq C \rho^{s-\alpha} r^\alpha.
$$
The third  term can be bounded by using again Theorem \ref{thm.bdry} together with
$$
d(x)\leq |y-\Pi(x)|\leq |y-x|+|x-\Pi(x)|\leq |y-x|+r \leq 2|x-y|, \quad \forall y\in B_r^c(x),
$$
to obtain 
\begin{align*}
\tailpp(u;x,r &)^{(p^+-1)\beta} \leq  r^{sp^+} C^{p^+-1} \int_{B_r^c}\frac{d^{s(p^+-1)(y)}}{|x-y|^{n+sp^+}}\,dy\\
&\leq  r^{sp^+} C^{p^+-1} \int_{B_r^c}\frac{|x-y|^{s(p^+-1)(y)}}{|x-y|^{n+sp^+}}\,dy\leq  r^{sp^+} C^{p^+-1} r^{s(p^+-1)}
\end{align*}
and the desired bound follows. The last term can be bounded analogously, and  the proof of claim \eqref{eq.in1} is completed. 

To prove the theorem, pick $x,y\in V$ and suppose without loss of generality that $|x-\Pi(x)|\geq |y-\Pi(y)|$. Two situations are possible: either $2|x-y|<|x-\Pi(x)|$, in which case we set $r=d(x)$ and apply \eqref{eq.in1} in $B_{r/2}(x)$ to get
$$
|u(x)-u(y)|\leq C   |x-y|^\alpha;
$$
or $2|x-y|\geq |x-\Pi(x)|\geq |y-\Pi(y)|$, in which case Theorem \ref{thm.bdry} ensures that
\begin{align*}
|u(x)-u(y)|&\leq |u(x)|+|u(y)|\leq C (d^s(x)+d^s(y))\\
&\leq C(|x-\Pi(x)|^s + |y-\Pi(y)|^s)\leq C(|x-y|^s)\leq C\rho^{s-\alpha}(|x-y|^\alpha).
\end{align*}
Therefore, the $\alpha-$H\"older seminorm is bounded in $V$, which concludes the proof.
\end{proof}

\appendix

\section{Some inequalities for Young functions}

We prove, for the reader's convenience, some technical results  used in the paper.

\begin{lem} \label{lema.0}
For every $a,b>0$ it holds that $g(a-b)-g(a)\leq -2^{1-\Lambda} g(b)$
\end{lem}
\begin{proof}
By using \eqref{delta2} and \eqref{g4} we get
\begin{align*}
g(b)&=g\left(2\frac{b-a+a}{2}\right) \leq 2^\Lambda g\left(\frac{b-a+a}{2}\right)\\
&\leq 2^{\Lambda-1}(g(b-a)+g(a)) = 2^{\Lambda-1}(g(a)- g(a-b))
\end{align*}
where in the last inequality we used that $g$ can be extended as an odd function.
\end{proof}

\begin{lem} \label{lema.0bis}
Fix $M>0$. There exists $C=C_M>0$ such that for every $|a|\leq M$ and $b>0$ it holds that
$g(a)-g(a-b)\leq C_M\max\{b,g(b)\}$.
\end{lem}

\begin{proof}
We separate two cases: if $b\leq M$ we have $g(a)-g(a-b)\leq |g'(M)||b|$ 
while if $b\geq M$ using \eqref{minmax1} we get $
g(a)-g(a-b)\leq g(M)+g(2M)\leq C_Mg(b)$.
\end{proof}

\begin{lem} \label{lema.separa}
For any $a,b\geq 0$ and $\theta\in(0,1)$ there exists $C_\theta$ such that $C_\theta\to \infty$ as $\theta\to 1^+$ and $g(a+b)\leq (1+\theta)^\Lambda g(a) +C_\theta g(b)
$.

\end{lem}

\begin{proof}
Given $\theta>0$ and $a,b>0$ (if either is equal to 0 the result is trivial). 

If $b>\theta a$, due to the monotonicity of $g$ and \eqref{delta2} we have, for $j_\theta\in\mathbb{N}$ large enough
\[
g(a+b)\leq g\left(\left(\frac{1}{\theta}+1\right)b\right)\leq g(2^{j_\theta}b)\leq 2^{j_\theta\Lambda}g(b).
\]
On the other hand, if $b\leq \theta a$ we get $g(a+b)\leq g\left((1+\theta)a\right)\leq (1+\theta)^\Lambda g(a)$ and the lemma is proved.
\end{proof}

\begin{lem} \label{delta2.inversa}
For all $a,b\geq 0$ it holds that $g^{-1}(a+b) \leq 2^\frac{1}{\lambda}(g^{-1}(a) + g^{-1}(b))$.
\end{lem}
\begin{proof}
It follows from the fact that $\frac{1}{\Lambda} \leq \frac{t(g^{-1})'(t)}{g^{-1}(t)}\leq \frac{1}{\lambda}$.
\end{proof}

Aside from the previous inequalities regarding Young functions, we will use a simple property of sets which are at positive distance from each other; recall that given $A,B\subset\R^n$ we define the distance between them as $\text{dist}(A,B):=\inf_{x\in A,y\in B}|x-y|$.

\begin{lem} \label{lemita}
If $A,B\subset \R^n$, with $A$ bounded and $dist(A,B^c)=d>0$, then
$$
|x-y|\geq C(A,B)(1+|y|), \qquad x\in A, y\in B^c.
$$
\end{lem}

\begin{proof}
Assume $A\subset B_R$ for some $R>0$ and set $C=C(A,B):=\frac{1+R}{d}$. Now,
\[
1+|y|\leq 1+|x|+|y-x|\leq 1+R+|x-y|= Cd+|x-y|\leq (1+C)|x-y|
\]
which gives the result.
\end{proof}

\section{Relation between weak, pointwise and strong solutions}
 
In this section show the  relation between weak, pointwise and strong solutions. The following lemma ensures that the definition of weak solution  makes sense:
\begin{lem}\label{lem.buenadefi}  	
Let $\Omega$ be a bounded domain in $\R^n$ and $u\in \W^{s,G}(\Omega)$. Define
\[
\langle\gfls u,\varphi\rangle:=\intr g\left(\frac{u(x)-u(y)}{|x-y|^s}\right)\frac{(\varphi(x)-\varphi(y))}{|x-y|^s}\,d\mu
\]
for $\varphi\in W_0^{s,G}(\Omega)$. Then the  $\langle\gfls u,\cdot\rangle\in W^{-s,\tilde{G}}(\Omega)$.
\end{lem}

\begin{proof}
Let $U\supset\supset\Omega$ such that 
\[
\|u\|_{W^{s,G}(U)}+\int_{\R^n} g\left(\frac{|u(x)|}{(1+|x|)^s}\right) \frac{dx}{(1+|x|)^{n+s}}<\infty.
\]

Now, for $\varphi\in W^{s,G}_0(\Omega)$ we can write
\begin{align*}
\langle\gfls u,\varphi\rangle & = \int\int_{U\times U} g\left(\frac{u(x)-u(y)}{|x-y|^s}\right)\frac{(\varphi(x)-\varphi(y))}{|x-y|^{n+s}}\,dxdy \\
							 & + 2\int\int_{\Omega\times U^c} g\left(\frac{u(x)-u(y)}{|x-y|^s}\right)\varphi(x)\frac{1}{|x-y|^{n+s}}\,dxdy := I_1+I_2.
\end{align*}

First, for $I_1$ we have
\begin{align*}
I_1 & \leq \int\int_{U\times U} \tilde{G}\left(g\left(\frac{u(x)-u(y)}{|x-y|^s}\right)\right)\:\frac{dxdy}{|x-y|^n}+\int\int_{U\times U} G\left(\frac{\varphi(x)-\varphi(y)}{|x-y|^s}\right)\frac{dxdy}{|x-y|^n} \\
	& \leq (p-1)\int\int_{U\times U} G\left(\frac{u(x)-u(y)}{|x-y|^s}\right)\:\frac{dxdy}{|x-y|^n}+\int\int_{U\times U} G\left(\frac{\varphi(x)-\varphi(y)}{|x-y|^s}\right)\frac{dxdy}{|x-y|^n}
\end{align*}
where in the last line we used Lemma 2.9 in \cite{FBS}. Then, $I_1$ is finite and continuous with respect to the strong convergence.

As for $I_2$, we compute, using Lemmas \ref{lema.separa} and \ref{lemita} and the fact that $g$ is increasing,
\begin{align*}
&\int_{U^c} g\left(\frac{u(x)-u(y)}{|x-y|^s}\right)\frac{1}{|x-y|^{n+s}}\,dy\leq \\ 
&\leq \left(\frac{3}{2}\right)^\Lambda\int_{U^c} g\left(\frac{|u(x)|}{|x-y|^s}\right)\frac{1}{|x-y|^{n+s}}\,dy 
+ C\int_{U^c} g\left(\frac{|u(y)|}{|x-y|^s}\right)\frac{1}{|x-y|^{n+s}}\,dy \\
& \leq C \left(g(|u(x)|)\int_{U^c} \frac{\max\{|x-y|^{-s\lambda},|x-y|^{-s\Lambda}\}}{|x-y|^{n+s}}\,dy 
 + \int_{\R^n} g\left(\frac{|u(y)|}{(1+|y|)^s}\right)\frac{1}{(1+|y|)^{n+s}}\,dy\right).
\end{align*}
Notice that $|x-y|$ on the first integral of the last term is bounded from below by dist$(U^c,\Omega)>0$, so both integrals are finite and we conclude the proof.
\end{proof}

The next lemma will be used to show that strong solutions are also weak solutions, but it will also be useful as stated below. First we need to recall the notion of Hausdorff distance between sets:
\[
d_H(A,B)=\max\left\{\sup_{x\in A}\text{dist}(x,B),\sup_{y\in B}\text{dist}(A,y)\right\}.
\]

\begin{lem}\label{lem.lim}  	
Let $u\in \W^{s,G}_{\text{loc}}(\Omega)$ and let $A_\varepsilon\subset\R^n\times\R^n$ be a neighborhood of ${\bf D}:=\{x=y\}$ such that 
\begin{itemize}
\item[(i)] $(x,y)\in A_\varepsilon$ then $(y,x)\in A_\varepsilon$;

\item[(ii)] $d_H(A_\varepsilon,{\bf D})\longrightarrow0$ as $\varepsilon\rightarrow0^+$.
\end{itemize}	
Consider for any $x\in \R^n$
\[
h_\varepsilon(x):=\int_{A_\varepsilon^c(x)}g\left(\frac{u(x)-u(y)}{|x-y|^s}\right)\frac{dy}{|x-y|^{n+s}}
\]
where $A_\varepsilon^c(x):=\{y\in\R^n:(x,y)\in A_\varepsilon\}$.

If $2h_\varepsilon\longrightarrow f$ in $L^1_\text{loc}(\Omega)$ as $\varepsilon\rightarrow0^+$ then $\gfls u=f$ weakly in $\Omega$. 
\end{lem}

\begin{proof}
We may assume $\Omega$ is bounded and that $U\supset\supset\Omega$ is such that 
\[
\|u\|_{W^{s,G}(U)}+\int_{\R^n} g\left(\frac{|u(x)|}{(1+|x|)^s}\right) \frac{dx}{(1+|x|)^{n+s}}<\infty.
\]
Further, by density it is enough to show that
\[
\int\int g\left(\frac{u(x)-u(y)}{|x-y|^s}\right)\frac{\varphi(x)-\varphi(y)}{|x-y|^s}\,d\mu=\int_{\Omega}f\varphi\,dx
\]
holds for any $\varphi\in C^\infty_c(\Omega)$. Let $\varphi$ be such and denote its support by $K$. 

Let us show first that $h_\varepsilon\in L^1(K)$. Notice that given $x\in K$ there exists $\rho>0$ such that $B_\rho(x)\subset A_\varepsilon(x)$ and that such $\rho$ can be taken independently of $x$ (but not necessarily of $\varepsilon$) via a covering argument. We can compute, similarly to Lemma \ref{lem.buenadefi},
\begin{align*}
\int_K &|h_\varepsilon(x)|\:dx  =\int_K\int_{A_\varepsilon^c(x)}g\left(\frac{u(x)-u(y)}{|x-y|^s}\right)\frac{dy}{|x-y|^{n+s}} \\
							& \leq C\left(\int_K\int_{A_\varepsilon^c(x)}g\left(\frac{|u(x)|}{|x-y|^s}\right)\frac{dy}{|x-y|^{n+s}}\:dx +  \int_K\int_{A_\varepsilon^c(x)}g\left(\frac{|u(y)|}{|x-y|^s}\right)\frac{dy}{|x-y|^{n+s}}\:dx\right) \\
							& \leq C\left(\int_K\int_{B_\rho^c}g\left(\frac{|u(x)|}{|x-y|^s}\right)\frac{dy}{|x-y|^{n+s}}\:dx +  |K|\int_{\R^n}g\left(\frac{|u(y)|}{(1+|y|)^s}\right)\frac{dy}{(1+|y|)^{n+s}}\:dx\right)<\infty.
\end{align*}
 
On the other hand, Lemma \ref{lem.buenadefi} shows that 
\[
g\left(\frac{u(x)-u(y)}{|x-y|^s}\right)\frac{\varphi(x)-\varphi(y)}{|x-y|^s}\in L^1(\R^n\times\R^n,d\mu)
\]
and therefore, by using our hypothesis we get that
\begin{align*}
\iint &g\left(\frac{u(x)-u(y)}{|x-y|^s}\right) \frac{\varphi(x)-\varphi(y)}{|x-y|^s}\frac{dxdy}{|x-y|^n}= \\
&= \lim_{\varepsilon\rightarrow0^+}\iint_{A_\varepsilon^c} g\left(\frac{u(x)-u(y)}{|x-y|^s}\right)\frac{\varphi(x)-\varphi(y)}{|x-y|^s}\frac{dxdy}{|x-y|^n} \\
&= \lim_{\varepsilon\rightarrow0^+}2\int_K\int_{A_\varepsilon^c(x)} g\left(\frac{u(x)-u(y)}{|x-y|^s}\right)\varphi(x)\frac{dydx}{|x-y|^{n+s}} = \lim_{\varepsilon\rightarrow0^+}2\int_K h_\varepsilon(x)\varphi(x)\:dx
\end{align*}
and the result follows since $2h_\varepsilon(x)\longrightarrow f$ in $L^1(K)$.
\end{proof}

\begin{cor}\label{cor.strongweak}	
Let $u\in \W^{s,G}_{\text{loc}}(\Omega)$ be a strong solution to $\gfls u=f$ in $\Omega$ with $f\in L^1_{\text{loc}}(\Omega)$. Then $u$ is also a weak solution. 
\end{cor}

\begin{proof}
The proof follows from Lemma \ref{lem.lim} with $A_\varepsilon=\{(x,y)\in\R^{2n}:|x-y|<\varepsilon\}$.
\end{proof}

\section{Properties of $\gfls$}

In this last section we prove some properties of the operator $\gfls$. We start with two lemmata regarding the behavior under scaling and rotation:  

\begin{lem}\label{lem.scaling}
Let $u$ be solution of 
\begin{align*}
\begin{cases}
(-\Delta_g)^s u =f &\quad \text{ in } B_R\\
u=0 &\quad \text{ in } B_R^c.
\end{cases}
\end{align*}
If we define, for $x\in \R^n$ and $t\geq 0$
$$
u^R(x)= u\left(Rx\right),\quad \tilde{f}(x)=f\left(Rx\right), \quad g_R(t)=  g\left(R^{-s}t\right), 
$$
then $u_R$ solves
\begin{align*}
\begin{cases}
(-\Delta_{g_R})^s u^R =\tilde{f} &\quad \text{ in } B_1\\
u^R=0 &\quad \text{ in } B_1^c.
\end{cases}
\end{align*}

In particular, if $|(-\Delta_g)^s u|\leq K$ in $B_R$, then $|(-\Delta_{g_R})^s u^R|\leq K$ in $B_1$.
\end{lem}

\begin{proof}
The proof that $u^R$ solves the equation is a straightforward change of variables $(\tilde{x},\tilde{y})=\left(Rx,Rx\right)$ (recall the definition of $\mu$) 
\begin{align*}
\left\langle(-\Delta_{g})^s u^R,\varphi\right\rangle & =\int\int g\left(R^{-s}\frac{|u\left(Rx\right)-u\left(Rx\right)|}{|x-y|^s}\right)\frac{u\left(Rx\right)-u\left(Rx\right)}{|u\left(Rx\right)-u\left(Rx\right)|}\frac{\varphi(x)-\varphi(y)}{|x-y|^s}\:d\mu \\
	 									   & =R^{-n}\int\int g\left(\frac{|u(\tilde{x})-u(\tilde{y})|}{|\tilde{x}-\tilde{y}|^s}\right)\frac{u(\tilde{x})-u(\tilde{y})}{|u(\tilde{x})-u(\tilde{y})|}\frac{\varphi\left(\frac{\tilde{x}}{R}\right)-\varphi\left(\frac{\tilde{y}}{R}\right)}{|\tilde{x}-\tilde{y}|^s}\:d\mu\\
	 									   & =R^{-n}\left\langle(-\Delta_g)^s u,\varphi\left(\frac{\cdot}{R}\right)\right\rangle.
\end{align*}
On the other hand
\[
\int \tilde{f}\varphi\:dx=\int f(Rx)\varphi(x)\:dx=R^{-n}\int f(\tilde{x})\varphi\left(\frac{\tilde{x}}{R}\right)\:d\tilde{x}
\]
and we conclude. 
\end{proof} 

\begin{rem}\label{rem.scaling}
Observe that our scaling preserves ellipticity; given $g$ satisfying \eqref{L} and $R>0$, the function $g_R(t)=g(R^{-s} t)$, $t\geq 0$ satisfies also \eqref{L}. Indeed, 
$$
\frac{t (g_R (t))'}{g_R(t)} = \frac{ \tau g'(\tau) }{g(\tau)} 
$$
where $\tau=R^{-s}t$. This simple remark is of paramount importance as it allows us to prove our estimates in, say, $B_1$ and obtain the general results by scaling.
\end{rem}

\begin{lem}\label{lem.rotation}
Let $u\in\W^{s,G}(\Omega)$ be a weak solution of $\gfls u=f$ in $\Omega$ for some $f\in L^1_{\text{loc}}(\Omega)$. 

Then for any orthogonal matrix $O\in\R^{n\times n},\: u_O(x):=u(Ox)\in\W^{s,G}(O^{-1}\Omega)$ and 
\[
\gfls u_O=f_O\quad\text{ weakly in }O^{-1}\Omega.
\]
\end{lem}

\begin{proof}
Let $(\tilde{x},\tilde{y})=\left(Ox,Ox\right)$ and change variables (recall orthogonal matrices preserve norms) 
\begin{align*}
\left\langle(-\Delta_g)^s u_O,\varphi\right\rangle & =\int\int g\left(\frac{|u\left(Ox\right)-u\left(Ox\right)|}{|x-y|^s}\right)\frac{u\left(Ox\right)-u\left(Ox\right)}{|u\left(Ox\right)-u\left(Ox\right)|}\frac{\varphi(x)-\varphi(y)}{|x-y|^s}\:d\mu \\
	 									   & =\int\int g\left(\frac{|u(\tilde{x})-u(\tilde{y})|}{|\tilde{x}-\tilde{y}|^s}\right)\frac{u(\tilde{x})-u(\tilde{y})}{|u(\tilde{x})-u(\tilde{y})|}\frac{\varphi\left(O^{-1}\tilde{x}\right)-\varphi\left(O^{-1}\tilde{y}\right)}{|O^{-1}\tilde{x}-O^{-1}\tilde{y}|^s}\,d\mu\\
	 									   & =\left\langle(-\Delta_g)^s u,\varphi\left(O^{-1}  \tilde y\right)\right\rangle.
\end{align*}
On the other hand
\[
\int_\Omega f_O\varphi\:dx=\int f(Ox)\varphi(x)\:dx=\int_{O^{-1}\Omega} f(\tilde{x})\varphi\left(O^{-1}\tilde{x}\right)\:d\tilde{x}
\]
and we conclude. 
\end{proof} 

Next, we prove a comparison principle for weak solutions.

\begin{prop}  \label{compara}
Let $\Omega$ be bounded, $u,v \in \widetilde{W}^{s,G}(\Omega)$  such that $u\leq v$ in $\Omega^c$ and
$$
\langle u,\varphi\rangle \leq  \langle v,\varphi\rangle \qquad \forall \varphi \in W^{s,p}_0(\Omega), \quad \varphi\geq 0  \text { in }\Omega.
$$
Then $u\leq v$ in $\Omega$.
\end{prop}

\begin{proof}
By hypothesis we have
\[
\intr g\left(\frac{u(x)-u(y)}{|x-y|^s}\right)\frac{\varphi(x)-\varphi(y)}{|x-y|^s}\,d\mu \leq\intr g\left(\frac{v(x)-v(y)}{|x-y|^s}\right)\frac{\varphi(x)-\varphi(y)}{|x-y|^s}\,d\mu
\]
for any $\varphi\in W^{s,G}_0(\Omega)$.

Subtracting the lhs of the previous inequality to the rhs and using 
\[
g(b)-g(a)=(b-a)\int_0^1g'(a+t(b-a))\:dt=(b-a)Q(x,y)
\]
for $b=\frac{v(x)-v(y)}{|x-y|^s}$ and $a=\frac{u(x)-u(y)}{|x-y|^s}$ 
\[
\intr Q(x,y)\frac{((u(y)-v(y)-(u(x)-v(x))(\varphi(x)-\varphi(y))}{|x-y|^{2s}}\:d\mu\geq 0
\]
with
\[
Q(x,y):=\int_0^1g'\left(\frac{u(x)-u(y)}{|x-y|^s}+t\left(\frac{v(x)-v(y)}{|x-y|^s}-\frac{u(x)-u(y)}{|x-y|^s}\right)\right)\:dt.
\]
Notice that $Q$ is nonnegative. 

Now use $\varphi=(u-v)_+$ as a test function and note that, calling $w=u-v$, 
\begin{align*}
((u(y)-v(y)-&(u(x)-v(x))(\varphi(x)-\varphi(y))  =-(w(x)-w(y))(w_+(x)-w_+(y)) \\
											   & =-(w_+(x)-w_+(y))^2 -w_-(x)w_+(y)-w_-(y)w_+(x) \leq 0.
\end{align*}

Therefore $Q(x,y)(w(x)-w(y))(\varphi(x)-\varphi(y))=0\quad\forall x,y\in\R^n$.

This can only happen if either of the terms vanishes, but in all three cases we get $(u-v)_+(x)=(u-v)_+(y)$, so this equality holds identically. Since outside $\Omega$ this gives $0$  we must have $(u-v)_+\equiv0$ in $\Omega$ as desired.
\end{proof}

Finally, the next Lemma is instrumental in several parts of the rest of the paper and it is strongly nonlocal in character:

\begin{lem} \label{nonlocal.behavior}
Let $u\in \widetilde W^{s,G}_{loc}(\Omega)$ be such that solves $(-\Delta_g)^s u=f$ (weakly, strongly, pointwisely) in $\Omega$  for some $f\in L^1_{loc}(\Omega)$. Let $v\in L^1_{loc}(\R^n)$ be such that
$$
dist(\supp(v),\Omega)>0, \qquad H_{g,\Omega^c}(u):=\int_{\Omega^c} g\left(\frac{u(x)}{(1+|x|)^s}\right) \frac{dx}{(1+|x|)^{n+s}}. <\infty,
$$
and define for a.e. Lebesgue point $x\in\Omega$ of $u$
$$
h(x)=2\int_{\supp (v)} \left[g\left(\frac{u(x) -u(y)-v(y)}{|x-y|^s} \right) - g\left(\frac{u(x) -u(y)}{|x-y|^s} \right) \right]\frac{ 1}{|x-y|^{n+s}} \,dy.
$$
Then $u+v \in \widetilde W_{loc}^{s,G}(\Omega)$ and it solves $(-\Delta_g)^s (u+v)=f+h$ (weakly, strongly, pointwisely) in $\Omega$. 
\end{lem}

\begin{proof}
It suffices to consider $\Omega$ bounded. Let us see that $u+v \in \widetilde W_{loc}^{s,G}(\Omega)$. Denote $K=\supp(v)$ and $U$ such that 
\[
\|u\|_{s,G,U} + \int_{\R^n} g\left(\frac{|u(x)|}{(1+|x|)^s}\right) \frac{dx}{(1+|x|)^{n+s}}<\infty,
\]
and without loss of generality that $\Omega\subset\subset U \subset\subset K^c$. It follows that $u+v=u$ in $U$ and it belongs to $W^{s,G}(U)$. Moreover, in light of \eqref{delta2}
\begin{align*}
\int_{\R^n} g\left(\frac{|u(x)+v(x)|}{(1+|x|)^s}\right) \frac{dx}{(1+|x|)^{n+s}} & \leq 2^\Lambda \left( \int_{\R^n} g\left(\frac{|u(x)|}{(1+|x|)^s}\right) \frac{dx}{(1+|x|)^{n+s}} \right. \\
																				& +\left. \int_K g\left(\frac{|v(x)|}{(1+|x|)^s}\right) \frac{dx}{(1+|x|)^{n+s}}\right),
\end{align*}
which is finite due to the assumptions on $u$ and $v$. Similarly, by using Lemma \ref{lemita} and the assumptions on $u$ and $v$ we get
$$
h(x)\leq C((-\Delta_g)^s u(x) + H_{g,K}(v))<\infty
$$
for some constant $C>0$ independent of $u$ and $v$.

Assume that $(-\Delta_g)^s u=f$ weakly. Let $\varphi\in C_c^\infty(\Omega)$. Then
\begin{align*}
\langle& (-\Delta_g)^s (u+v),\varphi\rangle = \iint_{\Omega\times\Omega} g\left(\frac{u(x) -u(y)}{|x-y|^s} \right) \frac{\varphi(x)-\varphi(y)}{|x-y|^s} \, d\mu+\\
&+\iint_{\Omega\times\Omega^c} g\left(\frac{u(x)-u(y) -v(y)}{|x-y|^s} \right) \frac{\varphi(x)}{|x-y|^s} \, d\mu-\iint_{\Omega^c\times\Omega}  g\left(\frac{ u(x)+v(x) -u(y) }{|x-y|^s} \right) \frac{\varphi(y)}{|x-y|^s} \, d\mu.
\end{align*}
The last expression can be written as
\begin{align*}
&\iint_{\R^n\times\R^n} g\left(\frac{u(x) -u(y)}{|x-y|^s} \right) \frac{\varphi(x)-\varphi(y)}{|x-y|^s} \, d\mu-
\iint_{\Omega\times\Omega^c} g\left(\frac{u(x) -u(y)}{|x-y|^s} \right) \frac{\varphi(x)}{|x-y|^s} \, d\mu\\
&-\iint_{\Omega^c\times\Omega} g\left(\frac{u(x) -u(y)}{|x-y|^s} \right) \frac{\varphi(y)}{|x-y|^s} \, d\mu
+2\iint_{\Omega\times\Omega^c} g\left(\frac{u(x) -u(y)-v(y)}{|x-y|^s} \right) \frac{\varphi(x)}{|x-y|^s} \, d\mu,
\end{align*}
and we obtain that
\begin{align*}
&\langle (-\Delta_g)^s  (u+v),\varphi\rangle = \int_\Omega f\varphi\,dx\; +\\
&+2\iint_{\Omega\times\Omega^c} \left[g\left(\frac{u(x) -u(y)-v(y)}{|x-y|^s} \right) - g\left(\frac{u(x) -u(y)}{|x-y|^s} \right) \right]\frac{\varphi(x)}{|x-y|^s} \, d\mu= \int_\Omega (f+h)\varphi\,dx
\end{align*}
where we have used Fubini's Theorem. The result follows by a density argument.

If we have now $\gfls u=f$ strongly or pointwisely in $\Omega$. Let $x\in V\subset\subset\Omega$ and $\varepsilon<$dist$(V,\Omega^c)$, and consider
\[
I_\varepsilon=\int_{B_\varepsilon^c(x)}g\left(\frac{u(x)+v(x)-u(y)-v(y)}{|x-y|^s}\right)\frac{dy}{|x-y|^{n+s}}.
\] 
We want to take $\lim_{\varepsilon\rightarrow0^+}I_\varepsilon$ to get the pointwise result. Since $v$ vanishes inside $\Omega$,
\begin{align*}
I_\varepsilon & =\int_{\Omega\setminus B_\varepsilon(x)}g\left(\frac{u(x)-u(y)}{|x-y|^s}\right)\frac{dy}{|x-y|^{n+s}}+\int_{\Omega^c}g\left(\frac{u(x)-u(y)-v(y)}{|x-y|^s}\right)\frac{dy}{|x-y|^{n+s}} \\
			  & = \int_{ B_\varepsilon(x)}g\left(\frac{u(x)-u(y)}{|x-y|^s}\right)\frac{dy}{|x-y|^{n+s}} \\
			  & \;\;+ \int_{K}\left(g\left(\frac{u(x)-u(y)-v(y)}{|x-y|^s}\right)-g\left(\frac{u(x)-u(y)}{|x-y|^s}\right)\right)\frac{dy}{|x-y|^{n+s}}
\end{align*}
so taking the limit gives the pointwise result. For the strong solution, we just need to be able to use Dominated Convergence Theorem; for that simply notice that 
\[
\int_{K}\left(g\left(\frac{u(x)-u(y)-v(y)}{|x-y|^s}\right)-g\left(\frac{u(x)-u(y)}{|x-y|^s}\right)\right)\frac{dy}{|x-y|^{n+s}}\in L^1(K)
\]
by a similar reasoning to that of the proof of Lemma \ref{lem.buenadefi}.
\end{proof}

\subsection*{Acknowledgements.} This work was partially supported by CONICET under grant  PIP No. 11220150100032CO, by ANPCyT under grant PICT 2016-1022 and by the University of Buenos Aires under grant 20020170100445BA. JFB and AS are members of CONICET and HV is a postdoctoral fellow of CONICET.

\end{document}